\documentclass[UTF-8,reqno]{amsart}
\usepackage{enumerate}
\usepackage{amssymb,url,color, booktabs}

\usepackage{mathrsfs}
\usepackage{amsmath}

\usepackage{fancyhdr}
\usepackage[nobysame]{amsrefs}
\BibSpec{article}{%
+{}{\PrintAuthors} {author}
+{,}{ \textrm} {title}
+{.}{ \textit} {journal}
+{,}{ \textbf} {volume}
+{}{ \parenthesize} {date}
+{,}{ } {pages}
+{.}{ arXiv:} {eprint}
+{.}{} {transition}
}
\BibSpec{book}{%
+{}{\PrintAuthors} {author}
+{,}{ \textit} {title}
+{.}{ \textrm} {series} 
+{,}{ Vol.} {volume}
+{.}{ } {publisher}
+{,}{ } {date}
+{.}{} {transition}
}
\usepackage{color}
\usepackage[colorlinks=true]{hyperref}
\hypersetup{
    linkcolor=blue,          
    citecolor=red,        
    filecolor=blue,      
    urlcolor=cyan
}

\numberwithin{equation}{section}

\newcommand{\be}{\begin{eqnarray}}
\newcommand{\ee}{\end{eqnarray}}
\newcommand{\ce}{\begin{eqnarray*}}
\newcommand{\de}{\end{eqnarray*}}
\newtheorem{theorem}{Theorem}[section]
\newtheorem{lemma}[theorem]{Lemma}
\newtheorem{remark}[theorem]{Remark}
\newtheorem{definition}[theorem]{Definition}
\newtheorem{proposition}[theorem]{Proposition}
\newtheorem{Examples}[theorem]{Example}
\newtheorem{corollary}[theorem]{Corollary}

\newenvironment{proof of theorem 1.2}{{\it Proof of Theorem 1.2}.}{{\hfill 	
$\square$\hskip - \parfillskip}}
\newenvironment{proof of theorem 1.3}{{\it Proof of Theorem 1.3}.}{{\hfill 	
		$\square$\hskip - \parfillskip}}
\newenvironment{proof of theorem 1.5}{{\it Proof of Theorem 1.5}.}{{\hfill 	
		$\square$\hskip - \parfillskip}}
\newenvironment{proof of theorem 1.4}{{\it Proof of Theorem 1.4}.}{{\hfill 	
		$\square$\hskip - \parfillskip}}
\newenvironment{proof of theorem 1.6}{{\it Proof of Theorem 1.6}.}{{\hfill 	
		$\square$\hskip - \parfillskip}}
\newenvironment{proof of theorem 5.3}{{\it Proof of Theorem 5.3}.}{{\hfill 	
		$\square$\hskip - \parfillskip}}
\newenvironment{proof of (1.3)}{{\it Proof of (1.3)}.}{{\hfill 	
		$\square$\hskip - \parfillskip}}
\newenvironment{proof of theorem 1.8}{{\it Proof of Theorem 1.8}.}{{\hfill 	
			$\square$\hskip - \parfillskip}}
\newenvironment{proofs of Theorem 1.1-1.5}{{\it Proofs of Theorem 1.1-1.5}.}{{\hfill 	
		$\square$\hskip - \parfillskip}}

\makeatletter
\newcommand{\rmnum}[1]{\romannumeral #1}
\newcommand{\Rmnum}[1]{\expandafter\@slowromancap\romannumeral #1@}
\makeatother

\def\eps{\varepsilon}

\def\om{\omega}
\def\Om{\Omega}

\def\p{\partial}

\def\l{\lambda}
\def\la{\langle}
\def\ra{\rangle}
\def\[{{\Big[}}
\def\]{{\Big]}}
\def\<{{\langle}}
\def\>{{\rangle}}
\def\({{\Big(}}
\def\){{\Big)}}

\def\bx{{\mathbf{x}}}

\def\min{{\mathord{{\rm min}}}}

\def\={&\!\!=\!\!&}

\def\cK{{\mathcal K}}

\def\mR{{\mathbb R}}
\def\mS{{\mathbb S}}

\def\1{{\mathbf{1}}}

\def\sA{{\mathscr A}}

\def\geq{\geqslant}
\def\leq{\leqslant}
\def\ge{\geqslant}
\def\le{\leqslant}

\def\k{\kappa}

\def\eps{\varepsilon}

\def\om{\omega}
\def\Om{\Omega}

\def\p{\partial}

\def\l{\lambda}
\def\la{\langle}
\def\ra{\rangle}
\def\[{{\Big[}}
\def\]{{\Big]}}
\def\<{{\langle}}
\def\>{{\rangle}}
\def\({{\Big(}}
\def\){{\Big)}}

\def\bx{{\mathbf{x}}}

\def\min{{\mathord{{\rm min}}}}

\def\={&\!\!=\!\!&}
\def\bt{\begin{theorem}}
\def\et{\end{theorem}}
\def\bl{\begin{lemma}}
\def\el{\end{lemma}}
\def\br{\begin{remark}}
\def\er{\end{remark}}
\def\bx{\begin{Examples}}
\def\ex{\end{Examples}}
\def\bd{\begin{definition}}
\def\ed{\end{definition}}
\def\bp{\begin{proposition}}
\def\ep{\end{proposition}}
\def\bc{\begin{corollary}}
\def\ec{\end{corollary}}

\def\geq{\geqslant}
\def\leq{\leqslant}
\def\ge{\geqslant}
\def\le{\leqslant}

 \def\nn{\nabla}

\def\<{\langle} \def\>{\rangle}

\def\bpf{\begin{proof}}
\def\epf{\end{proof}}

\allowdisplaybreaks

\begin{document}
	
\title{A class of anisotropic inverse Gauss curvature flows and dual Orlicz Minkowski type problem}\thanks{\it {This research was partially supported by NSFC (No. 11871053).}}
\author{Shanwei Ding*$^1$ and Guanghan Li$^1$}


\thanks{{\it *Corresponding author. E-mail: 2014301000108@whu.edu.cn}}

\thanks{{\it Guanghan Li: ghli@whu.edu.cn}}

\thanks{\it $^1$ School of Mathematics and Statistics, Wuhan University, Wuhan 430072, China.}

\thanks{\it Declarations of interest: none.}

\begin{abstract}
In this paper, we study the long-time existence and asymptotic behavior for a class of anisotropic inverse Gauss curvature flows. By the stationary solutions of anisotropic flows, we obtain some new existence results for the dual Orlicz Minkowski type problem and even dual Orlicz Minkowski type problem for smooth measures, which is the most reasonable extension of the $L^p$  dual Minkowski problem from the dual point of view.  The results of corresponding $L^p$ versions are $L^p$ dual Minkowski problem for $p>q$; and even $L^p$ dual Minkowski problem for $p>-1$, or $q<1$, or some ranges of $p<0<q$, which contain all existence results for smooth measures up to now except $p=q$ or $q=n+1$ ($L^p$ Minkowski problem).

\end{abstract}

\maketitle
\setcounter{tocdepth}{2}
\tableofcontents

\section{Introduction}

Anisotropic inverse curvature flows with strictly convex hypersurfaces and speed depending on their curvatures, support function and radial function have been considered recently, cf.  \cite{IM,DL3,BIS,CL} etc.. These flows usually provided alternative proofs and smooth category approach of the existence of solutions to elliptic PDEs arising in convex body geometry. However, the literature on non-homogeneous anisotropic inverse curvature  flows is not very rich and there are few works in this direction, cf. \cite{JLL,BIS2,BIS3,DL4}.  One advantage of this method is that there is no need to employ the constant rank theorem. For example, if $F=\sigma_{k}^{\frac{1}{k}}(\l)$, by the lower bound of $F^{-1}=\sigma_{k}^{-\frac{1}{k}}(\l)=(\frac{\sigma_{n}}{\sigma_{n-k}})^\frac{1}{k}(\k)$ one has that the hypersurfaces naturally preserve convexity, where $\k$ and $\l=\frac{1}{\k}$ are the principal curvature and principal curvature radii of hypersurfaces respectively. Whether these flows can be extended is an interesting problem. In our previous work \cite{DL4}, we use anisotropic flows without global terms to derive some results for dual Orlicz Christoffel-Minkowski type problem. In that paper, we can set the constant in elliptic PDEs to $1$ if some conditions are satisfied.  In this paper firstly we want to use anisotropic flows with global terms to derive some results for these problems.  However, if $F\ne \sigma_{n}^\frac{1}{n}(\l)$, there is no suitable monotone integral quantities to be used to prove the convergence of flows if speed depends on  $X$. More details can be seen in \cite{DL3}; If $F=\sigma_{n}^\frac{1}{n}(\l)$, inspired by the homogeneous case \cite{CL,DL3}, we consider a large of anisotropic flows with global terms.

Let $M_0$ be a closed, smooth and strictly convex hypersurface in $\mathbb{R}^{n+1}$ ($n\geq2$), and $M_0$ encloses the origin. In this paper, we study the following expanding flow
   \begin{equation}
 	\label{1.1}
 	\begin{cases}
 		&\frac{\partial X}{\partial t}(\cdot,t)=(\eta(t)\varphi(\nu,u)G(X)\sigma_{n}^{\frac{\beta}{n}}(\l_i)-1)u\nu,\\
 		&X(\cdot,0)=X_0,
 	\end{cases}
 \end{equation}
where $\sigma_{n}(\l_i)=\prod_{i=1}^n\l_i$ is the $n$-th elementary symmetric function of the hypersurface $M_t$ parameterized by a smooth embedding $X(\cdot,t): \mS^n\times[0,T^*)\to \mR^{n+1}$, $\l=(\l_1,\cdots,\l_n)$ are the principal curvature radii of the  hypersurface $M_t$, $\beta>0$, $\varphi:\mS^n\times(0,+\infty)\rightarrow(0,+\infty)$ and $G:\mR^{n+1}\backslash\{0\}\rightarrow(0,+\infty)$ are two smooth functions, $u$ is the support function, $\nu(\cdot,t)$ is the outer unit normal vector field to $M_t$ and
\begin{equation}\label{1.2}
\eta(t)=\dfrac{\int_{\mS^n}G(\xi,\rho)^\frac{n}{\beta}\rho^{n+1}d\xi}{\int_{\mS^n} u\varphi(x,u)G^{\frac{n}{\beta}+1}(X)\sigma_{n}^{\frac{\beta}{n}+1}dx}.
\end{equation}
Note that $\rho=\sqrt{\langle X,X\rangle}$, $X=\rho\xi$, thus $G(X)$ can be regarded as $G(\xi,\rho)$. $X$ can be also denoted by $X:=\bar\nn u=ux+Du$,  where $\langle\cdot,\cdot\rangle$ is the standard inner product in $\mR^{n+1}$,  $D$ is the covariant derivative with respect to an orthonormal frame on $\mS^n$, $\bar\nn$ is the covariant derivative with respect to the metric in Euclidean space.

In the definition of $\eta(t)$, we use the following notations. Let $\Om_{t}$ be the convex body enclosed by $M_t$. Firstly we mention that the definitions of the support function $u=u(x):\mS^n\rightarrow\mR$ and radial function $\rho=\rho(\xi):\mS^n\rightarrow\mR$ of  a convex body which contains the origin are equal to the definitions of the support function and radial function of the hypersurface enclosing the convex body respectively. Let $\vec{\rho}(\xi):=\rho(\xi)\xi$. We then introduce two set-valued mappings, the radial Gauss mapping $\mathscr{A}=\sA_\Om$ and the reverse radial Gauss mapping $\sA^*=\sA^*_\Om$, which are given by, for any Borel set $\om\subset\mS^n$,
\begin{align*}
	\sA(\om)=&\{\nu(\vec{\rho}(\xi)):\xi\in\om\},\\
	\sA^*(\om)=&\{\xi\in\mS^n:\nu(\vec{\rho}(\xi))\in\om\}.
\end{align*}
Note that $\sA(\xi)$ (resp. $\sA^*(x)$) is a unique vector for almost all $\xi\in\mS^n$ (resp. for almost all $x\in\mS^n$) \cite{SR}. Thus we shall use the notations for convenience
\begin{equation*}
	dx=d\mu_{\mS^n}(x) \text{ and }d\xi=d\mu_{\mS^n}(\xi).
\end{equation*}
By the dual body, we find an interesting relation about $\varphi$ and $G$. Let $\Om_{t}$ be the convex body whose support function is $u(\cdot,t)$, $\Om_{t}^*$ be its polar body
$$\Om_{t}^*=\{z\in\mR^{n+1}:z\cdot y\le1 \text{ for all } y\in \Om_{t}\},$$
and $u^*(\cdot,t)$ be the support function of $\Om_{t}^*$. If $\rho(\cdot,t)$ is the radial function of $\Om_{t}$, then
\begin{equation}\label{3.10}
	u^*(\xi,t)=\frac{1}{\rho(\xi,t)}.
\end{equation}
It is well-known that $\sA_{\Om_t^*}=\sA_{\Om_t}^*$, see e.g. \cite{HLY2}. Similar to $X=\rho\xi$, we can consider the vector $Y:=\frac{1}{u}\nu$. Since $\frac{1}{u}=\sqrt{\langle Y,Y\rangle}$, $\varphi(\nu,u)$ can be regarded as $\varphi(Y)$. In fact, the position vector of dual body is $\rho^*(\xi^*)\xi^*=\frac{1}{u}\nu=Y$. Thus, if there is a existence result under a condition of $\varphi$, there must be a existence result under the corresponding condition of $G$. This is the key observation to our main results. 
 More details can be seen in the proof of Lemma \ref{l3.3}.

Flow (\ref{1.1}) is inspired by our previous work \cite{DL3}. However, flow (\ref{1.1}) is more complicated than the flow in \cite{DL3}, since it involves two nonlinear functions $\varphi$ and $G$. Note $G$ is a function of $X$ other than $\vert X\vert$, which needs more effort to deal with. 

In this paper, we study the general and even stationary solutions of anisotropic flows. A function $f:\mS^n\rightarrow \mR$ is called even if
$$f(-x)=f(x),\qquad \forall x\in\mS^n.$$
Thus $\varphi(x,u)$ is called even if
$$\varphi(-x,u)=\varphi(x,u),\qquad \forall x\in\mS^n.$$
Note that a support function is even if and only if the convex body determined by this support function is origin symmetric. Let $B_r$ be the ball with radius $r$ in $\mR^{n+1}$. Let
\begin{equation}\label{1.3}
	\begin{split}
		q^*&=\begin{cases}
			\dfrac{q}{q-n} \quad \text{  if } q\ge n+1,\\[2pt]
			\dfrac{nq}{q-1} \quad \text{  if } 1<q<n+1,\\[2pt]
			+\infty \qquad \text{  if } 0<q\le1.
		\end{cases}
	\end{split}
\end{equation}

We first prove the long time existence and convergence of the flow (\ref{1.1}). In the following theorems, we also denote $\varphi(Y)=\varphi(\nu,u)$, where $Y\in\mR^{n+1}\backslash\{0\}$.
\begin{theorem}\label{t1.1}
Let $M_0$ be a closed, smooth, uniformly convex hypersurface in $\mathbb{R}^{n+1}$, $n\ge2$, enclosing the origin. Let $\varphi:\mS^n\times(0,+\infty)\rightarrow(0,+\infty)$ and $G:\mR^{n+1}\backslash\{0\}\rightarrow(0,+\infty)$ be two smooth functions. Suppose the following conditions hold.

(\rmnum{1}) If $\exists\eps,\delta>0$, $\forall \theta_1,\theta_2\in\mS^n$, 
$$\int_{1}^\infty\int_{\{\la x,\theta_1\ra\ge\eps\}\cap\mS^n}\varphi^{-\frac{n}{\beta}}(x,s)dx ds=+\infty, \int_{0}^1\int_{\{\la\xi,\theta_2\ra\ge\delta\}\cap\mS^n}G(r\xi)^\frac{n}{\beta}r^nd\xi dr=+\infty;$$

(\rmnum{2}) either

(a) $\int_{B_1}\varphi^{-\frac{n}{\beta}}(y)|y|^{-n}dy<+\infty$,

or (b) $\int_{\mR^{n+1}\backslash B_1}G(y)^\frac{n}{\beta}dy<+\infty$.

\noindent Then flow (\ref{1.1}) has a unique smooth strictly convex solution $M_t$ for all time $t>0$ and a subsequence of $M_t$ converges in $C^\infty$-topology to a positive, smooth, uniformly convex solution to $\varphi(\nu,u)G(X)\sigma_{n}^{\frac{\beta}{n}}=c$ for some $c>0$.
\end{theorem}
\textbf{Remark:} (1) The conditions (\rmnum{1}) and (\rmnum{2}b) is mentioned for the first time in \cite{GHW1} if $\varphi(x,u)=f(x)g(u)$. Compared with \cite{GHW1}, we need more effort to deal with general $\varphi(x,u)$. Condition (\rmnum{2}a) is new.

(2) Since $\int_{B_1}\varphi^{-\frac{n}{\beta}}(y)|y|^{-n}dy=\int_{0}^1\int_{\mS^n}\varphi^{-\frac{n}{\beta}}(x,s)dxds$ and $\int_{\mR^{n+1}\backslash B_1}G(y)^\frac{n}{\beta}dy=\int_{1}^\infty\int_{\mS^n}G(r\xi)^\frac{n}{\beta}r^nd\xi dr$,  conditions (\rmnum{1}) and (\rmnum{2}) are similar in form. Conditions (\rmnum{2}a) and (\rmnum{2}b) seem not similar because of the form in equation (\ref{1.7}). If we consider the homogeneous case, i.e. $\varphi(x,s)^\frac{n}{\beta}=f(x)s^{1-p}$ and $G(X)^\frac{n}{\beta}=|X|^{q-n-1}$, we can find that the integrands in conditions (\rmnum{2}a) and (\rmnum{2}b) should differ by $|y|^{-n}$ by comparison.

(3) When $\varphi(x,s)^\frac{n}{\beta}=f(x)s^{1-p}$ and $G(X)^\frac{n}{\beta}=|X|^{q-n-1}$, assumptions (\rmnum{1}), (\rmnum{2}a) and (\rmnum{1}), (\rmnum{2}b) are satisfied for $p>0\ge q$ and $p\ge0> q$ respectively.

\begin{theorem}\label{t1.2}
Let $M_0$ be a closed, smooth, origin-symmetric, uniformly convex hypersurface in $\mathbb{R}^{n+1}$, $n\ge2$, enclosing the origin. Suppose $\varphi:\mS^n\times(0,+\infty)\rightarrow(0,+\infty)$ and $G:\mR^{n+1}\backslash\{0\}\rightarrow(0,+\infty)$ are two smooth and even functions.

Suppose either


(\rmnum{1}) $\int_{\mR^{n+1}\backslash B_1}\varphi^{-\frac{n}{\beta}}(y)|y|^{-n}dy=+\infty$, and for some $\hat{C}_1>0$, we have
$$\int_{\mS^n}\int_{|\la x,\theta\ra|}^1\varphi^{-\frac{n}{\beta}}(x,s)dsdx\le\hat{C}_1, \quad \forall\theta\in\mS^n;$$

or (\rmnum{2}) $\int_{\mR^{n+1}\backslash B_1}\varphi^{-\frac{n}{\beta}}(y)|y|^{-n}dy<+\infty$, $\int_{ B_1}\varphi^{-\frac{n}{\beta}}(y)|y|^{-n}dy=+\infty$, and for some $\hat{C}_2>0$, we have
$$\int_{\mS^n}\int_{|\la x,\theta\ra|}^\infty\varphi^{-\frac{n}{\beta}}(x,s)dsdx\le\hat{C}_2, \quad \forall\theta\in\mS^n;$$


or (\rmnum{3}) $\int_{B_1}G(y)^\frac{n}{\beta}dy=+\infty$, and for some $\hat{C}_3>0$, we have
$$\int_{\mS^n}\int^\frac{1}{|\la\xi,\theta\ra|}_1G(r\xi)^\frac{n}{\beta}r^ndrd\xi\le\hat{C}_3, \quad \forall\theta\in\mS^n;$$

or (\rmnum{4}) $\int_{B_1}G(y)^\frac{n}{\beta}dy<+\infty$, $\int_{\mR^{n+1}\backslash B_1}G(y)^\frac{n}{\beta}dy=\infty$, and for some $\hat{C}_4>0$, we have
$$\int_{\mS^n}\int^\frac{1}{|\la\xi,\theta\ra|}_0G(r\xi)^\frac{n}{\beta}r^ndrd\xi \le\hat{C}_4, \quad \forall\theta\in\mS^n;$$

or (\rmnum{5}) $\int_{\mR^{n+1}\backslash B_1}\varphi^{-\frac{n}{\beta}}(y)|y|^{-n}dy<+\infty$, and for some $-q^*<p<0<q$, $\hat{C}_5>0$ and $\forall\rho(\xi)>0\in C^\infty(\mS^n)$, we have
$${\lim\sup}_{s\rightarrow0^+}\frac{\int_{s}^\infty\varphi^{-\frac{n}{\beta}}(x,t)dt}{s^p}<\infty,\qquad \int_{\mS^n}\int^{\rho(\xi)}_0G(r\xi)^\frac{n}{\beta}r^ndrd\xi\le \hat{C}_5\int_{\mS^n}\rho^q(\xi)d\xi;$$

or (\rmnum{6}) $\int_{B_1}G(y)^\frac{n}{\beta}dy<+\infty$, and for some $-q^*<p<0<q$, $\hat{C}_6>0$ and $\forall u(x)>0\in C^\infty(\mS^n)$, we have
$${\lim\sup}_{s\rightarrow0^+}\frac{\int_{0}^\frac{1}{s}G(r\xi)^\frac{n}{\beta}r^ndr}{s^p}<\infty,\qquad \int_{\mS^n}\int_\frac{1}{u(x)}^\infty\varphi^{-\frac{n}{\beta}}(x,s)dsdx\le \hat{C}_6\int_{\mS^n}u^q(x)dx.$$
In the case (\rmnum{2}), we additionally choose an origin-symmetric initial hypersurface $M_0$ such that
$$\int_{\mS^n}\int_{u_{M_0}}^\infty\varphi^{-\frac{n}{\beta}}(x,s)dsdx>\hat{C}_2.$$
In the case (\rmnum{4}), we additionally choose an origin-symmetric initial hypersurface $M_0$ such that
$$\int_{\mS^n}\int^{\rho_{M_0}}_0G(r\xi)^\frac{n}{\beta}r^ndrd\xi >\hat{C}_4.$$
Then flow (\ref{1.1}) has a unique smooth strictly convex origin-symmetric solution $M_t$ for all time $t>0$ and a subsequence of $M_t$ converges in $C^\infty$-topology to a positive, smooth, uniformly convex origin-symmetric solution to $\varphi(\nu,u)G(X)\sigma_{n}^{\frac{\beta}{n}}=c$ for some $c>0$.
\end{theorem}

\textbf{Remark:} (1) When $\varphi(x,s)^\frac{n}{\beta}=f(x)s^{1-p}$ and $G(X)^\frac{n}{\beta}=|X|^{q-n-1}$, assumptions (\rmnum{1})-(\rmnum{4}) are satisfied, in order, for $p\ge0$, $-1<p<0$, $q\le0$, $0<q<1$ respectively; both assumptions (\rmnum{5}) and (\rmnum{6}) are satisfied for $-q^*<p<0<q$.   To see this, note that
\begin{align*}
&\int_{\mS^n}\log|x_i|d\mu>-\infty,\\
&\int_{\mS^n}\frac{1}{|x_i|^l}d\mu<\infty \quad \text{ for } 0<l<1,
\end{align*}
where $(x_1,\cdots,x_{n+1})$ denotes the Euclidean coordinates. In summary, we derive the even results with $p>-1$, or $q<1$ or $-q^*<p<0<q$. Hence, this theorem generalizes many well-known results about flows to the Orlicz setting, such as \cite{CL,DL3,BIS}. Meanwhile, this theorem covers Theorem 2.1 in \cite{BIS3}. Comparing with \cite{BIS3}, the speed of the flow in this paper has a more complex form since $G=1$ and $\varphi(x,u)=f(x)g(u)$ in \cite{BIS3}. We derive this results with more complex form without additional conditions.

(2) The long time existence and convergence result of flows corrsponding to
\begin{equation*}
\int_{0}^1\varphi^{-\frac{n}{\beta}}(x,s)ds<+\infty \text{ and } \int_{1}^\infty\varphi^{-\frac{n}{\beta}}(x,s)ds=+\infty,\quad \forall x\in\mS^n
\end{equation*}
can be deduced from \cite{BIS3}, i.e. $G=1$ and $\varphi(x,u)=f(x)g(u)$ in that paper. In fact, we can derive the long time existence and convergence result under a weaker condition
\begin{equation}\label{x1.5}
\int_{0}^1\int_{\mS^n}\varphi^{-\frac{n}{\beta}}(x,s)dxds<+\infty \text{ and } \int_{1}^\infty\int_{\mS^n}\varphi^{-\frac{n}{\beta}}(x,s)dxds=+\infty.
\end{equation}
For general $\varphi(x,u)$, this situation is harder to deal with. In this theorem, we derive a weaker condition (\rmnum{1}), i.e. the condition (\rmnum{1}) is weaker than (\ref{x1.5}). Thus we need to do lots of effort to deal with this situation. By the dual relation, we can derive a long time existence and convergence  result if
\begin{equation}\label{x1.6}
\int_{B_1}G(y)^\frac{n}{\beta}dy=+\infty \text{ and } \int_{\mR^{n+1}\backslash B_1}G(y)^\frac{n}{\beta}dy<+\infty.
\end{equation}
This result is covered by case (\rmnum{3}).
In condition (\ref{x1.5}), we have
$$\int_{\mS^n}\int_{0}^{|\la x,\theta\ra|}\varphi^{-\frac{n}{\beta}}(x,s)dsdx<\int_{0}^1\int_{\mS^n}\varphi^{-\frac{n}{\beta}}(x,s)dxds<\infty, \quad \forall\theta\in\mS^n.$$
Therefore, the integral conditions in (\rmnum{1}) and (\rmnum{2}) compared to the condition (\ref{x1.5}) are not more restrictive. Similarly, the integral conditions in (\rmnum{3}) and (\rmnum{4}) are also not more restrictive. We mention that the cases (\rmnum{1}), (\rmnum{2}) and (\rmnum{6}) are derived by  (\rmnum{3}), (\rmnum{4}) and (\rmnum{5}) respectively and the dual relation. Note that $\int_{\mS^n}\int^{\rho(\xi)}_0G(r\xi)^\frac{n}{\beta}r^ndrd\xi\le \hat{C}_5\int_{\mS^n}\rho^q(\xi)d\xi$ in case (\rmnum{5}) is weaker than $\int^{\rho}_0G(r\xi)^\frac{n}{\beta}r^ndr\le \hat{C}_5\rho^q$.

(3) There must have a lot of hypersurfaces satisfying the cases (\rmnum{2}) and (\rmnum{4}). In fact, spheres with radius small enough or large enough satisfy the cases (\rmnum{2}) and (\rmnum{4}) by $\int_{0}^\infty\int_{\mS^n}\varphi^{-\frac{n}{\beta}}(x,s)dsdx=+\infty$ and $\int_{\mR^{n+1}\backslash B_1}G(y)^\frac{n}{\beta}dy=\infty$ respectively. These conditions are only used to ensure the existence of initial hypersurfaces that satisfy these cases.

Under flow (\ref{1.1}), if we parameterize the hypersurface $M_t$ by inverse Gauss map $X(x,t): \mS^n\times[0,T^*)\to \mR^{n+1}$, by \cite{DL} Section 2 the support function $u$ satisfies
\begin{equation}\label{1.4}
	\begin{cases}
		&\frac{\partial u}{\partial t}=(\eta(t)\varphi(\nu,u)G(X)\sigma_{n}^{\frac{\beta}{n}}(\l_i)-1)u,\\
		&u(\cdot,0)=u_0.
	\end{cases}
\end{equation}

Next, we slightly modify flow (\ref{1.1}) to treat a class of regular dual Orlicz Minkowski type problems without the assumption that $\varphi$ and $G$ are even and we can set $c=1$ in this situation. Inspired by \cite{DL4}, we consider the flow
   \begin{equation}
	\label{x1.7}
	\begin{cases}
		&\frac{\partial X}{\partial t}(\cdot,t)=(\varphi(\nu,u)G(X)\sigma_{n}^{\frac{\beta}{n}}(\l_i)-1)u\nu,\\
		&X(\cdot,0)=X_0,
	\end{cases}
\end{equation}
i.e. $\eta=1$ in flow (\ref{1.1}). Therefore we don't need to distinguish between flow (\ref{1.1}) and (\ref{x1.7}) usually.
\begin{theorem}\label{xt1.3}
Let $M_0$ be a closed, smooth, uniformly convex hypersurface in $\mathbb{R}^{n+1}$, $n\ge2$, enclosing the origin. Let $\varphi:\mS^n\times(0,+\infty)\rightarrow(0,+\infty)$ and $G:\mR^{n+1}\backslash\{0\}\rightarrow(0,+\infty)$ be two smooth functions. Suppose there are two positive constant $r_1<r_2$ such that
\begin{equation}\label{1.8}
\begin{cases}
	&\varphi(x,r_1)G(r_1x)r_1^\beta\ge1,\quad \text{ for}\quad\forall x\in \mS^n,\\
	&\varphi(x,r_2)G(r_2x)r_2^\beta\le1,\quad \text{ for}\quad\forall x\in \mS^n.
\end{cases}
\end{equation}
If $r_1<\rho_{M_0}<r_2$, flow (\ref{x1.7}) has a unique smooth strictly convex solution $M_t$ for all time $t>0$ and a subsequence of $M_t$ converge in $C^\infty$-topology to a positive, smooth, uniformly convex solution to $\varphi(\nu,u)G(X)\sigma_{n}^{\frac{\beta}{n}}=1$.
\end{theorem}
\textbf{Remark:} (1) The condition in this theorem is weaker than the condition in Theorem 1.4 in \cite{DL4}. However, in \cite{DL4}, arbitrarily initial, strictly convex hypersurface has a unique smooth strictly convex solution $M_t$ for all time and a subsequence converge in $C^\infty$-topology to a positive, smooth, uniformly convex solution to $\varphi(\nu,u)G(X)\sigma_{n}^{\frac{\beta}{n}}=1$ along flow (\ref{x1.7}). In this theorem we give a restriction on initial hypersurface. Compared with Theorem \ref{t1.1}, we can set $c=1$ in this theorem.

(2) When $\varphi(x,s)^\frac{n}{\beta}=f(x)s^{1-p}$ and $G(X)^\frac{n}{\beta}=|X|^{q-n-1}$, the condition (\ref{1.8}) is satisfied for $p>q$.

The asymptotic behavior of the flow means existence of solutions of the dual Orlicz Minkowski type problem mentioned below. In order to prove the above theorems, we shall establish the a priori estimates for the parabolic equation (\ref{1.4}).

Convex geometry plays important role in the development of fully nonlinear partial differential equations. The classical Minkowski problem, the Christoffel-Minkowski problem, the $L^p$ Minkowski problem, the $L^p$ Christoffel-Minkowski problem in general, are beautiful examples of such interactions (e.g., \cite{FWJ2,GM,LE,LO}). Very recently, the $L^p$ dual curvature measures \cite{LYZ2,HLY2} were developed to the Orlicz case \cite{GHW1,GHW2}, which
unified more curvature measures and were named dual Orlicz curvature measures.
These curvature measures are of central importance in the dual Orlicz–Brunn
Minkowski theory, and the corresponding Minkowski problems are called the dual
Orlicz–Minkowski problems. This problem can be reduced to derive the solution of 
\begin{equation}\label{1.5}
	c\varphi(u)G(\bar\nn u)\det(D^2u+u\Rmnum{1})=f(x), \quad \forall x\in\mS^n.
\end{equation}
There have a few results, e.g. \cite{GHW1,GHW2,LL,CLL,DL4}.

We shall consider the dual Orlicz Minkowski type problem, which is the most reasonable extension of the $L^p$  dual Minkowski problem from the dual point of view. As before, this problem can be reduced to the following nonlinear PDE:
\begin{equation}\label{1.6}
\varphi(x,u)G(\bar\nn u)\det(D^2u+u\Rmnum{1})=c\quad\text{ on } \mS^n.
\end{equation}
(\ref{1.6}) can be converted into $L^p$ Minkowski problem if $\varphi(x,u)G(\bar\nn u)=\psi(x)u^{1-p}$, or $L^p$ dual Minkowski problem if  $\varphi(x,u)G(\bar\nn u)=\psi(x)u^{1-p}\rho^{q-n-1}$, or $L^p$ Aleksandrov problem if   $\varphi(x,u)G(\bar\nn u)=\psi(x)u^{1-p}\rho^{-n-1}$, or 
dual Orlicz Minkowski problem  if   $\varphi(x,u)G(\bar\nn u)=\psi(x)\varphi(u)\phi(X)$ respectively. Some previous known results of these problems can be seen in \cite{BIS,LYZ2,CL,HZ,CCL,CHZ,BF,DL3,BLY,Z} etc..

By Theorem \ref{t1.1}, \ref{t1.2} and \ref{xt1.3}, we have the following existence results of the dual Orlicz Minkowski type problem and even dual Orlicz Minkowski type problem by the asymptotic behavior of the anisotropic curvature flow.
\begin{theorem}\label{t1.3}
Let $\varphi:\mS^n\times(0,+\infty)\rightarrow(0,+\infty)$ and $G:\mR^{n+1}\backslash\{0\}\rightarrow(0,+\infty)$ be two smooth functions.

(1) If $\exists\eps,\delta>0$, $\forall \theta_1,\theta_2\in\mS^n$, 
$$\int_{1}^\infty\int_{\{\la x,\theta_1\ra\ge\eps\}\cap\mS^n}\varphi^{-1}(x,s)dx ds=+\infty, \int_{0}^1\int_{\{\la\xi,\theta_2\ra\ge\delta\}\cap\mS^n}G(r\xi)r^nd\xi dr=+\infty;$$

Suppose either

(a) $\int_{B_1}\varphi^{-1}(y)|y|^{-n}dy<+\infty$,

or (b) $\int_{\mR^{n+1}\backslash B_1}G(y)dy<+\infty$.

\noindent Then there exists a  positive, smooth, uniformly convex solution to  (\ref{1.6}) for some positive constant $c$.

(2) If there are two positive constant $r_1<r_2$ such that
\begin{equation*}
	\begin{cases}
		&\varphi(x,r_1)G(r_1x)r_1^n\ge1,\quad \text{ for}\quad\forall x\in \mS^n,\\
		&\varphi(x,r_2)G(r_2x)r_2^n\le1,\quad \text{ for}\quad\forall x\in \mS^n.
	\end{cases}
\end{equation*}
Then there exists a  positive, smooth, uniformly convex solution to  (\ref{1.6}) with $c=1$.
\end{theorem}
\textbf{Remark}: (1)  Note that  the power of curvature function is different between this theorem and Theorem \ref{t1.1}, \ref{xt1.3}. Thus $\varphi$ and $G$ are different between this theorem and Theorem \ref{t1.1}, \ref{xt1.3}.

(2) This theorem covers some well-known non-homogeneous results, e.g. \cite{GHW1,FHL,LL}. If $\varphi(x,u)=f(x)u^{1-p}$ and $G(X)=|X|^{q-n-1}$, the problem can be converted to the $L^p$ dual Minkowski problem, i.e.
\begin{equation}\label{1.7}
	\sigma_{n}(D^2u+u\Rmnum{1})=\psi(x)u^{p-1}(u^2+\vert Du\vert^2)^\frac{n+1-q}{2}.
\end{equation}
We derive existence results of this problem with $p>q$. Thus this theorem also covers some well-known homogeneous results, e.g. \cite{HLY,Z}, partial results in \cite{CL,DL3,DL4,HZ}. We mention that there only exists regular existence results of $L^p$ dual Minkowski problem for $p\ge q$ in smooth measure up to now, i.e. the solutions are non-even. We generalize the $L^p$ version to Orlicz version. Since we can do more processing in homogeneous case, there exists the existence result for $p=q$ in homogeneous case. Thus it is understandable that we don't get this result in non-homogeneous case.

\begin{theorem}\label{t1.4}
Let $\varphi:\mS^n\times(0,+\infty)\rightarrow(0,+\infty)$ and $G:\mR^{n+1}\backslash\{0\}\rightarrow(0,+\infty)$ be two smooth and even functions.

Suppose either

(\rmnum{1}) $\int_{\mR^{n+1}\backslash B_1}\varphi^{-1}(y)|y|^{-n}dy=+\infty$, and for some $\hat{C}_1>0$, we have
$$\int_{\mS^n}\int_{|\la x,\theta\ra|}^1\varphi^{-1}(x,s)dsdx\le\hat{C}_1, \quad \forall\theta\in\mS^n;$$

or (\rmnum{2}) $\int_{\mR^{n+1}\backslash B_1}\varphi^{-1}(y)|y|^{-n}dy<+\infty$, $\int_{ B_1}\varphi^{-1}(y)|y|^{-n}dy=+\infty$, and for some $\hat{C}_2>0$, we have
$$\int_{\mS^n}\int_{|\la x,\theta\ra|}^\infty\varphi^{-1}(x,s)dsdx\le\hat{C}_2, \quad \forall\theta\in\mS^n;$$


or (\rmnum{3}) $\int_{B_1}G(y)dy=+\infty$, and for some $\hat{C}_3>0$, we have
$$\int_{\mS^n}\int^\frac{1}{|\la\xi,\theta\ra|}_1G(r\xi)r^ndrd\xi\le\hat{C}_3, \quad \forall\theta\in\mS^n;$$

or (\rmnum{4}) $\int_{B_1}G(y)dy<+\infty$, $\int_{\mR^{n+1}\backslash B_1}G(y)dy=\infty$, and for some $\hat{C}_4>0$, we have
$$\int_{\mS^n}\int^\frac{1}{|\la\xi,\theta\ra|}_0G(r\xi)r^ndrd\xi \le\hat{C}_4, \quad \forall\theta\in\mS^n;$$

or (\rmnum{5}) $\int_{\mR^{n+1}\backslash B_1}\varphi^{-1}(y)|y|^{-n}dy<+\infty$, and for some $-q^*<p<0<q$ and $\hat{C}_5>0$, we have
$${\lim\sup}_{s\rightarrow0^+}\frac{\int_{s}^\infty\varphi^{-1}(x,t)dt}{s^p}<\infty,\qquad \int_{\mS^n}\int^{\rho(\xi)}_0G(r\xi)r^ndrd\xi\le \hat{C}_5\int_{\mS^n}\rho^q(\xi)d\xi;$$

or (\rmnum{6}) $\int_{B_1}G(y)dy<+\infty$, and for some $-q^*<p<0<q$ and $\hat{C}_6>0$, we have
$${\lim\sup}_{s\rightarrow0^+}\frac{\int_{0}^\frac{1}{s}G(r\xi)r^ndr}{s^p}<\infty,\qquad \int_{\mS^n}\int_\frac{1}{u(x)}^\infty\varphi^{-1}(x,s)dsdx\le \hat{C}_6\int_{\mS^n}u^q(x)dx.$$

\noindent Then there exists a  positive, smooth, uniformly convex even solution to  (\ref{1.6}) for some positive constant $c$.
\end{theorem}
\textbf{Remark}: (1)  Note that  the power of curvature function is different between this theorem and Theorem \ref{t1.2}. Thus $\varphi$ and $G$ is different between this theorem and Theorem \ref{t1.2}.

(2) There have very little results about anisotropic non-homogeneous Gauss curvature or inverse Gauss curvature flows with global terms, e.g. \cite{BIS3,CLL}. The existence result with even Orlicz Minkowski problem corrsponding to
\begin{equation*}
	\int_{0}^1\varphi^{-1}(x,s)ds<+\infty \text{ and } \int_{1}^\infty\varphi^{-1}(x,s)ds=+\infty
\end{equation*}
can be deduced from \cite{BBC,HLYZ,BIS3}, i.e. $G=1$ and $\varphi(x,u)=f(x)g(u)$ in these papers. In fact, we can derive the existence result under a weaker condition
\begin{equation}\label{x1.12}
	\int_{0}^1\int_{\mS^n}\varphi^{-1}(x,s)dxds<+\infty \text{ and } \int_{1}^\infty\int_{\mS^n}\varphi^{-1}(x,s)dxds=+\infty.
\end{equation}
For general $\varphi(x,u)$, this situation is harder to deal with. In this theorem, we can find that the condition (\rmnum{1}) is weaker than (\ref{x1.12}). Thus we need to do lots of effort to deal with this situation. By the dual relation, we can derive a existence result with even dual Orlicz Minkowski type problem if
\begin{equation}\label{x1.13}
\int_{B_1}G(y)dy=+\infty \text{ and } \int_{\mR^{n+1}\backslash B_1}G(y)dy<+\infty.
\end{equation}
This case can be seen in \cite{CLL}. However in \cite{CLL} they need an extra condition to derive the existence result and we drop the extra condition. (\ref{x1.13}) is covered by case (\rmnum{3}).
In condition (\ref{x1.12}), we have
$$\int_{\mS^n}\int_{0}^{|\la x,\theta\ra|}\varphi^{-1}(x,s)dsdx<\int_{0}^1\int_{\mS^n}\varphi^{-1}(x,s)dxds<\infty \quad \forall\theta\in\mS^n.$$
Therefore, the integral conditions in (\rmnum{1}) and (\rmnum{2}) compared to the condition (\ref{x1.12}) are not more restrictive. Similarly, the integral conditions in (\rmnum{3}) and (\rmnum{4}) are also not more restrictive.

(3) The cases (\rmnum{1})-(\rmnum{2}) are inspired by \cite{BIS3}, where they derive the existence results about even Orlicz Minkowski problem. We generalize $1$ and $f(x)\varphi(u)$ to general $G$ and $\varphi(x,u)$ respectively and don't need any extra conditions. The cases (\rmnum{1})-(\rmnum{6}) yield new existence results for the regular even dual Orlicz Minkowski problem. To the best of our knowledge, the conditions in cases (\rmnum{3}), (\rmnum{4}), and (\rmnum{6}) are new in the study of the dual Orlicz Minkowski problem.

(4) If $\varphi(x,u) G(X)=f(x)u^{1-p}|X|^{q-n-1}$, the problem can be converted to the $L^p$ dual Minkowski problem (\ref{1.7}), i.e.
\begin{equation*}
\sigma_{n}(D^2u+u\Rmnum{1})=\psi(x)u^{p-1}(u^2+\vert Du\vert^2)^\frac{n+1-q}{2}.
\end{equation*}
We derive even existence results of this problem with $p>-1$, or $q<1$, or $-q^*<p<0<q$. We mention that so far there is no even existence result about the $L^p$ dual Minkowski problem for $-1<p<0$ or $0<q<1$ if the generalized Blaschke-Santal$\acute{o}$ inequality isn't used. In \cite{CHZ} they derived even existence results about the $L^p$ dual Minkowski problem for $p\ge0$ or $q\le0$.
In summary, we derive the existence results of $L^p$ dual Minkowski problem in smooth measure for $p>q$ and even $L^p$ dual Minkowski problem for $p>-1$, or $q<1$, or $-q^*<p<0<q$. Our argument covers lots of well-known results, cf. \cite{BBC,BIS,BIS3,CHZ,CLL,HLYZ,HLY,FH,BLY,HZ}, partial results in \cite{DL3,CL,LSW} etc.. To the best of our knowledge, 
 the results in this paper covers all existence results of $L^p$ dual Minkowski problem and even $L^p$ dual Minkowski problem in smooth measure up to now except $p=q$ or $q=n+1$ ($L^p$ Minkowski problem).

The rest of the paper is organized as follows. We first recall some notations and known results in Section 2 for later use. In Section 3, we  establish the $C^0$ estimates.
In Section 4, we have the a priori estimates since we had derived the bounds of $\sigma_{n}$ and $C^2$ estimates for a large ranges of parabolic equations in \cite{DL4}. Then we show the convergence of this flow and complete the proof of these theorems.
\section{Preliminary}
\subsection{Intrinsic curvature}
We now state some general facts about hypersurfaces, especially those that can be written as graphs. The geometric quantities of ambient spaces will be denoted by $(\bar{g}_{\alpha\beta})$, $(\bar{R}_{\alpha\beta\gamma\delta})$ etc., where Greek indices range from $0$ to $n$. Quantities for $M$ will be denoted by $(g_{ij})$, $(R_{ijkl})$ etc., where Latin indices range from $1$ to $n$.

Let $\nabla$, $\bar\nabla$ and $D$ be the Levi-Civita connection of $g$, $\bar g$ and the Riemannian metric $e$ of $\mathbb S^n$  respectively. All indices appearing after the semicolon indicate covariant derivatives. The $(1,3)$-type Riemannian curvature tensor is defined by
\begin{equation*}
	R(U,Y)Z=\nabla_U\nabla_YZ-\nabla_Y\nabla_UZ-\nabla_{[U,Y]}Z,
\end{equation*}
or with respect to a local frame $\{e_i\}$,
\begin{equation*}
	R(e_i,e_j)e_k={R_{ijk}}^{l}e_l,
\end{equation*}
where we use the summation convention (and will henceforth do so).  Ricci identities read
\begin{equation*}
	Y_{;ij}^k-Y_{;ji}^k=-{R_{ijm}}^kY^m
\end{equation*}
for all vector fields $Y=(Y^k)$. We also denote the $(0,4)$ version of the curvature tensor by $R$,
\begin{equation*}
	R(W,U,Y,Z)=g(R(W,U)Y,Z).
\end{equation*}
\subsection{Extrinsic curvature}
The induced geometry of $M$ is governed by the following relations. The second fundamental form $h=(h_{ij})$ is given by the Gaussian formula
\begin{equation*}
	\bar\nabla_ZY=\nabla_ZY-h(Z,Y)\nu,
\end{equation*}
where $\nu$ is a local outer unit normal field. Note that here (and in the rest of the paper) we will abuse notation by disregarding the necessity to distinguish between a vector $Y\in T_pM$ and its push-forward $X_*Y\in T_p\mathbb{R}^{n+1}$. The Weingarten endomorphism $A=(h_j^i)$ is given by $h_j^i=g^{ki}h_{kj}$, and the Weingarten equation
\begin{equation*}
	\bar\nabla_Y\nu=A(Y),
\end{equation*}
holds there, or in coordinates
\begin{equation*}
	\nu_{;i}^\alpha=h_i^kX_{;k}^\alpha.
\end{equation*}
We also have the Codazzi equation in $\mathbb{R}^{n+1}$
\begin{equation*}
	\nabla_Wh(Y,Z)-\nabla_Zh(Y,W)=-\bar{R}(\nu,Y,Z,W)=0
\end{equation*}
or
\begin{equation*}
	h_{ij;k}-h_{ik;j}=-\bar R_{\alpha\beta\gamma\delta}\nu^\alpha X_{;i}^\beta X_{;j}^\gamma X_{;k}^\delta=0,
\end{equation*}
and the Gauss equation
\begin{equation*}
	R(W,U,Y,Z)=\bar{R}(W,U,Y,Z)+h(W,Z)h(U,Y)-h(W,Y)h(U,Z)
\end{equation*}
or
\begin{equation*}
	R_{ijkl}=\bar{R}_{\alpha\beta\gamma\delta}X_{;i}^\alpha X_{;j}^\beta X_{;k}^\gamma X_{;l}^\delta+h_{il}h_{jk}-h_{ik}h_{jl},
\end{equation*}
where
\begin{equation*}
	\bar{R}_{\alpha\beta\gamma\delta}=0.
\end{equation*}

\subsection{Convex hypersurface parametrized by the inverse Gauss map}
Let $M$ be a smooth, closed and uniformly convex hypersurface in $\mR^{n+1}$. Assume that $M$ is parametrized by the inverse Gauss map $X: \mS^n\to M\subset \mR^{n+1}$ and encloses origin. The support function $u: \mS^n\to \mR^1$ of $M$ is defined by $$u(x)=\sup_{y\in M}\la x,y\ra.$$ The supremum is attained at a point $y=X(x)$ because of convexity, $x$ is the unit outer normal of $M$ at $y$. Hence $u(x)=\la x,X(x)\ra$. Then we have
\begin{equation*}
X=Du+ux\quad \text{    and    } \quad\rho=\sqrt{u^2+|Du|^2}.
\end{equation*}

The second fundamental form of $M$ in terms of $u$ is given by
$$h_{ij}=D_iD_ju+ue_{ij},$$
and the principal radii of curvature are the eigenvalues of the matrix
\begin{equation}
	b_{ij}=h^{ik}g_{jk}=h_{ij}=D_{ij}^2u+u\delta_{ij}.
\end{equation}
The proof can be seen in Urbas \cite{UJ}. We see that the support function satisfies the initial value problem by \cite{DL} Section 2,
\begin{equation}\label{2.2}
	\begin{cases}
		&\frac{\p u}{\p t}=\(\eta(t)\varphi(\nu,u)G(ux+Du)\sigma_{n}^{\frac{\beta}{n}}([D^2u+u\uppercase\expandafter{\romannumeral1}])-1\)u\; \text{ on } \mS^n\times[0,\infty),\\
		&u(\cdot,0)=u_0,
	\end{cases}
\end{equation}
where $\uppercase\expandafter{\romannumeral1}$ is the identity matrix, $u_0$ is the support function of $M_0$, and
\begin{equation}
	F([a_{j}^i])=f(\mu_1,\cdots,\mu_n),
\end{equation}
where $\mu_1,\cdots,\mu_n$ are the eigenvalues of matrix $[a^i_{j}]$. It is not difficult to see that the eigenvalues of $[F^{j}_i]=[\frac{\p F}{\p a^i_{j}}]$ are $\frac{\p f}{\p\mu_1},\cdots,\frac{\p f}{\p\mu_n}$.  We also have the formula
\begin{equation}\label{2.4}
\frac{\p_t\rho}{\rho}(\xi)=\frac{\p_tu}{u}(x),
\end{equation}
which can be also found in \cite{LSW,CL}.

It is well-known that the determinants of the Jacobian of radial Gauss mapping $\sA$ and reverse radial Gauss mapping $\sA^*$ of $\Om$ are given by, see e.g. \cite{HLY2,CCL,LSW},
\begin{equation}\label{2.5}
	\vert Jac\sA\vert(\xi)=\vert\frac{dx}{d\xi}\vert=\frac{\rho^{n+1}(\xi)K(\vec{\rho}(\xi))}{u(\sA(\xi))},
\end{equation}
and
\begin{equation}\label{2.6}
	\vert Jac\sA^*\vert(x)=\vert\frac{d\xi}{dx}\vert=\dfrac{u(x)}{\rho^{n+1}(\sA^*(x))K(\nu^{-1}_\Om(x))}.
\end{equation}

Before closing this section, we give the following basic properties for any given $\Om$ contained the origin.
\begin{lemma}\cite{CL}\label{l2.1}
Let $\Om$ contain the origin. Let $u$ and $\rho$ be the support function and radial function of $\Om$, and $x_{\max}$ and $\xi_{\min}$ be two points such that $u(x_{\max})=\max_{\mS^n}u$ and $\rho(\xi_{\min})=\min_{\mS^n}\rho$. Then
\begin{align}
\max_{\mS^n}u=\max_{\mS^n}\rho\quad \text{    and    }\quad \min_{\mS^n}u=\min_{\mS^n}\rho,\label{2.7}\\
u(x)\ge x\cdot x_{\max}u(x_{\max}),\quad \forall x\in\mS^n,\label{2.8}\\
\rho(\xi)\xi\cdot\xi_{\min}\le\rho(\xi_{\min}),\quad \forall \xi\in\mS^n.\label{2.9}
\end{align}
\end{lemma}

We shall use the following generalized Blaschke-Santal$\acute{o}$ inequality. It was first proved in \cite{CHD}. In \cite{CCL} they gave a simple proof.
\begin{lemma}(Blaschke-Santal$\acute{o}$-type inequality \cite{CHD})\label{l2.2} Given $q>0$, let $q^*>0$ be the number given by (\ref{1.3}). For $s\in(0,q^*]$, $s\ne+\infty$, there is a constant $C_{n,q,s}>0$ such that
	\begin{equation*}
		\(\int_{\mS^n}\rho^q_\Om d\sigma_{\mS^n}\)^\frac{1}{q}\(\int_{\mS^n}\rho^s_{\Om^*} d\sigma_{\mS^n}\)^\frac{1}{s}\le C_{n,q,s} \quad \text{ for all } \Om\in\cK^e_0,
	\end{equation*}
	where $\Om^*$ is the dual body of $\Om$, $\cK^e_0$ is the set of all origin-symmetric convex bodies containing the origin in their interiors.
\end{lemma}

\section{Uniform bounds of support function}
In this section we will establish the $C^0$ estimate along flow (\ref{1.1}). Since  condition (\rmnum{2}) in Theorem \ref{t1.1} can be seen in Theorem \ref{t1.2}, we define the following quantities by the classification in Theorem \ref{t1.2}.

Define
\begin{equation}\label{3.1}
J(t)=\begin{cases}
		\int_{\mS^n}\int^{C_0}_u\varphi^{-\frac{n}{\beta}}(x,s)dsdx, &\text{ if cases  (\rmnum{2}), (\rmnum{5}) and (\rmnum{6}) hold,}\\
		\int_{\mS^n}\int_{C_0}^u\varphi^{-\frac{n}{\beta}}(x,s)dsdx, &\text{ if cases (\rmnum{1}), (\rmnum{3}) and (\rmnum{4}) hold,}
	\end{cases}
\end{equation}
\begin{equation}\label{3.2}
V(t)=\begin{cases}
\int_{\Om_t}G(y)^{\frac{n}{\beta}}dy, &\text{ if cases (\rmnum{4})-(\rmnum{6}) hold,}\\
\int_{B_{R}\backslash \Om_t}G(y)^{\frac{n}{\beta}}dy-\int_{\Om_t\backslash B_R}G(y)^{\frac{n}{\beta}}dy, &\text{ if cases (\rmnum{1})-(\rmnum{3}) hold,}
	\end{cases}
\end{equation}
where $C_0$ and $R$ are determined later. Through a simple calculation, we also have
\begin{equation}\label{3.3}
	V(t)=\begin{cases}
		\int_{\mS^n}\int_0^\rho G(r\xi)^{\frac{n}{\beta}}r^ndrd\xi, &\text{ if cases (\rmnum{4})-(\rmnum{6}) hold,}\\
		\int_{\mS^n}\int_\rho^R G(r\xi)^{\frac{n}{\beta}}r^ndrd\xi, &\text{ if cases (\rmnum{1})-(\rmnum{3}) hold.}
	\end{cases}
\end{equation}

We give the monotonicity of these quantities defined above.
\begin{lemma}\label{l3.1}
Along flow (\ref{1.1}), $V(t)$ remains unchanged. $J(t)$ is non-increasing if the cases  (\rmnum{1}), (\rmnum{3}) and (\rmnum{4}) hold. $J(t)$ is non-decreasing if the cases (\rmnum{2}), (\rmnum{5}) and (\rmnum{6}) hold. $J(t)$ remains unchanged if and only if $M_t$ satisfies $\varphi(x,u)G(X)\sigma_{n}^\frac{\beta}{n}=c.$
\end{lemma}
\begin{proof}
Along flow (\ref{1.1}), by (\ref{1.2}), (\ref{2.2}), (\ref{2.4}) and (\ref{2.6}), we have
\begin{align*}
\frac{d}{dt}V(t)=&\pm\int_{\mS^n}G(\rho\xi)^{\frac{n}{\beta}}\rho^n\p_t\rho(\xi,t) d\xi\\
=&\pm\int_{\mS^n}G(\rho\xi)^{\frac{n}{\beta}}\rho^{n+1}\frac{\p_tu(x,t)}{u} d\xi\\
=&\pm\int_{\mS^n}G(\bar\nn u)^{\frac{n}{\beta}}\sigma_{n}\p_tu(x,t) dx\\
=&\pm\(\int_{\mS^n}u\eta(t)\varphi(x,u)G^{\frac{n}{\beta}+1}(\bar\nn u)\sigma_{n}^{\frac{\beta}{n}+1}dx-\int_{\mS^n}G(\bar\nn u)^{\frac{n}{\beta}}\sigma_{n}udx\)\\
=&0.
\end{align*}
Here $+$ is for the cases (\rmnum{4})-(\rmnum{6}), and $-$ is for the cases (\rmnum{1})-(\rmnum{3}).
\begin{equation}\label{3.4}
\begin{split}
\frac{d}{dt}J(t)=&\pm\int_{\mS^n}\varphi^{-\frac{n}{\beta}}\p_tudx\\
=&\pm\(\eta(t)\int_{\mS^n}u\varphi^{1-\frac{n}{\beta}}G\sigma_{n}^{\frac{\beta}{n}}dx-\int_{\mS^n}u\varphi^{-\frac{n}{\beta}}dx\)\\
=&\pm\(\int_{\mS^n}u\varphi(x,u)G^{\frac{n}{\beta}+1}\sigma_{n}^{\frac{\beta}{n}+1}dx\)^{-1}\\
&\(\int\varphi G\sigma_{n}^{\frac{\beta}{n}}d\mu\int(\varphi G\sigma_{n}^{\frac{\beta}{n}})^\frac{n}{\beta}d\mu-\int d\mu\int(\varphi G\sigma_{n}^{\frac{\beta}{n}})^{\frac{n}{\beta}+1}d\mu\),
\end{split}
\end{equation}
where we denote $d\mu=u\varphi^{-\frac{n}{\beta}}dx$.
Here $+$ is for the cases cases (\rmnum{1}), (\rmnum{3}) and (\rmnum{4}), and $-$ is for the cases  (\rmnum{2}), (\rmnum{5}) and (\rmnum{6}). Let $Q=\varphi G\sigma_{n}^{\frac{\beta}{n}}$. By H$\ddot{o}$lder inequality, we have
\begin{align*}
\int Qd\mu\le \(\int d\mu\)^\frac{n}{n+\beta}\(\int Q^{\frac{n}{\beta}+1}d\mu\)^\frac{\beta}{n+\beta},\\
\int Q^\frac{n}{\beta}d\mu\le \(\int d\mu\)^\frac{\beta}{n+\beta}\(\int Q^{\frac{n}{\beta}+1}d\mu\)^\frac{n}{n+\beta}.
\end{align*}
All equalities hold if and only if $Q=c$.
Thus, by (\ref{3.4}) we completes the proof of this lemma.
\end{proof}

After these preparations, we now derive the uniform bounds of $u$. To use the maximum principle, we suppose $M(t):=\max u_{M_t}$ is attained at $e_{n+1}\in\mS^n$ and $m(t):=\min \rho_{M_t}$ is attained at $e_{1}\in\mS^n$ in the following proofs. Firstly we show the lower bound of $u$ for the cases (\rmnum{3}) and (\rmnum{4}) of Theorem \ref{t1.2}.
\begin{lemma}\label{l3.2}
	Under flow (\ref{1.1}), the corresponding assumptions of Theorem \ref{t1.2} and the cases (\rmnum{3}) and (\rmnum{4}), there exists a positive constant $C_1$ independent of $t$, such that
	$$u(x,t)\ge \frac{1}{C_{1}}.$$
	It also means that
	$$\rho\ge \frac{1}{C_{1}}$$
	by Lemma \ref{l2.1}.
\end{lemma}
\begin{proof}
Case (\rmnum{3}): Recall (\ref{3.2}) and let $R=1$.
We begin with the estimate about $\int_{\Om_t\backslash B_1}G(y)^{\frac{n}{\beta}}dy$ first. Let $S=\{\xi|\rho\xi\in \Om_t\backslash B_1\}$, i.e. $S=\{\xi\in\mS^n|\rho(\xi)>1\}$. Suppose $m<1$. Since $G$ and $\varphi$ are even, by the origin-symmetry of $M_0$, $M_t$ is origin-symmetric. Thus due to convexity and Lemma \ref{l2.1},
$$\rho\le \frac{m}{|\la\xi,e_1\ra|}.$$
Then
\begin{equation}\label{3.7}
\begin{split}
\int_{\Om_t\backslash B_1}G(y)^{\frac{n}{\beta}}dy=&\int_S\int_1^\rho G(y)^{\frac{n}{\beta}}r^ndrd\xi\\
\le&\int_S\int_1^{\frac{m}{|\la\xi,e_1\ra|}} G(y)^{\frac{n}{\beta}}r^ndrd\xi\\
\le&\int_{S}\int_1^{\frac{1}{|\la\xi,e_1\ra|}} G(y)^{\frac{n}{\beta}}r^ndrd\xi\\
\le&\int_{\mS^n}\int_1^{\frac{1}{|\la\xi,e_1\ra|}} G(y)^{\frac{n}{\beta}}r^ndrd\xi\\
\le&\hat{C_3},
\end{split}
\end{equation}
by $\int_1^{\frac{1}{|\la\xi,e_1\ra|}} G(y)^{\frac{n}{\beta}}r^ndr>0$ and the conditions in the case (\rmnum{3}) of Theorem \ref{t1.2}. By Lemma \ref{l2.1} and (\ref{3.7}), we have
\begin{equation*}
\begin{split}
C\ge V(t)=\int_{B_{1}\backslash\Om_t}G(y)^{\frac{n}{\beta}}dy-\int_{\Om_t\backslash B_1}G(y)^{\frac{n}{\beta}}dy\ge\int_{B_{1}\backslash \Om_t}G(y)^{\frac{n}{\beta}}dy-\hat{C_3}.
\end{split}
\end{equation*}
In other words,
\begin{equation}\label{3.8}
\int_{B_{1}\backslash \Om_t}G(y)^{\frac{n}{\beta}}dy\le C.
\end{equation}
Next,  We shall estimate $\int_{B_1\backslash \Om_t}G(y)^{\frac{n}{\beta}}dy$. 
\begin{equation}\label{3.5}
	\begin{split}
\int_{B_1\backslash \Om_t}G(y)^{\frac{n}{\beta}}dy=\int_{B_1}G(y)^{\frac{n}{\beta}}\chi_{B_1\backslash \Om_t}(y) dy,
	\end{split}
\end{equation}
where $\chi$ denotes the characteristic function. Suppose to the contrary that there exists a sequence of times $t_k$ such that
$$\min_{x\in\mS^n}u(x,t_k)\rightarrow0^+ \text{  as  }k\rightarrow\infty.$$
Let $L_k=B_1\cap \Om_{t_k}$. Then each $L_k\subset B_1$ is an origin-symmetric convex body. By the Blaschke selection theorem, $\{L_k\}$ has a subsequence which converges to a nonempty, compact, convex set $L$. Without loss of generality, we assume
$$L_k\rightarrow L \text{  as  }k\rightarrow\infty,$$
where $L\subset B_1$ is an origin-symmetric convex subset. Then
\begin{align*}
	\min_{\mS^n}u_L=&\lim_{k\rightarrow\infty}\min_{\mS^n}u_{L_k}\\
	\le&\lim_{k\rightarrow\infty}\min_{\mS^n}u(\cdot,t_k)\\
	=&0,
\end{align*}
which together with that $L$ is origin-symmetric implies that $L$ is contained in a lower dimension subspace, i.e. a hyperplane in $\mR^{n+1}$. Then
$$\chi_L=0 \quad \text{ a.e. in }B_1.$$
Thus, for a.e. $y\in B_1$, there is 
\begin{align*}
	\lim_{k\rightarrow\infty}\chi_{B_1\backslash \Om_{t_k}}(y)=&\lim_{k\rightarrow\infty}[1-\chi_{L_k}(y)]\\
	=&1-\chi_L(y)\\
	=&1.
\end{align*} 
By Lemma \ref{l3.1}, (\ref{3.8}), (\ref{3.5}) and Fatou lemma,
\begin{equation}\label{3.6}
	\begin{split}
		C\ge&{\lim\inf}_{k\rightarrow\infty}\int_{B_1}G(y)^{\frac{n}{\beta}}\chi_{B_1\backslash \Om_{t_k}}(y) dy\\
		\ge&\int_{B_1}{\lim\inf}_{k\rightarrow\infty}G(y)^{\frac{n}{\beta}}\chi_{B_1\backslash \Om_{t_k}}(y) dy\\
		=&\int_{B_1}G(y)^{\frac{n}{\beta}}dy\\
		=&+\infty,
	\end{split}
\end{equation}
which is a contradiction. Thus we derive the lower bound of $u$.

Case (\rmnum{4}): Suppose $m<1$. Similar to case (\rmnum{3}), we have
$$\rho\le \frac{m}{|\la\xi,e_1\ra|}\le\frac{1}{|\la\xi,e_1\ra|}.$$
Since $\int_{B_1}G(y)^{\frac{n}{\beta}}dy<\infty$, recall (\ref{3.3}) and our choice of $M_0$ in case (\rmnum{4}) of Theorem \ref{t1.2}, we have
\begin{equation}\label{3.9}
\begin{split}
\hat{C_4}<V(0)\le V(t)=\int_{\mS^n}\int_0^\rho G(r\xi)^{\frac{n}{\beta}}r^ndrd\xi\le\int_{\mS^n}\int_0^\frac{1}{|\la\xi,e_1\ra|} G(r\xi)^{\frac{n}{\beta}}r^ndrd\xi\le\hat{C_4},
\end{split}
\end{equation}
which is a contradiction. Thus we derive the lower bound of $u$. Since we shall use (\ref{3.9}) in the rest paper, we only use $V(0)\le V(t)$ in (\ref{3.9}) although $V(0)=V(t)$ here.
\end{proof}
If we have the lower bound of $u$ we can directly derive the upper bound of $u$ under flow (\ref{1.1}).
\begin{lemma}\label{l3.3}
Let $M_0$ be a closed, smooth, origin-symmetric, uniformly convex hypersurface in $\mathbb{R}^{n+1}$, $n\ge2$, enclosing the origin. Suppose $\varphi:\mS^n\times(0,+\infty)\rightarrow(0,+\infty)$ and $G:\mR^{n+1}\backslash\{0\}\rightarrow(0,+\infty)$ are two smooth and even functions. If $u(\cdot,t)\ge\frac{1}{C_1},$  along flow (\ref{1.1}), there exists a positive constant $C_2$ independent of $t$, such that
	$$u(x,t)\le  C_{2}.$$
		It also means that
	$$\rho,|Du|\le C_{2}$$
	by Lemma \ref{l2.1}.
\end{lemma}
\begin{proof}
We shall get a dual relation between $\varphi(x,u)$ and $G(\xi,\rho)$ in this proof. The definition of the dual body and some properties can be seen in (\ref{3.10}) and near (\ref{3.10}).

We divide this proof into two cases. 

(1) If $\int_1^\infty\int_{\mS^n}\varphi(x,s)^{-\frac{n}{\beta}}dxds=+\infty$. Let $$J(t)=\int_{\mS^n}\int_\frac{1}{C_1}^u\varphi^{-\frac{n}{\beta}}(x,s)dsdx.$$
By the calculation of the proof of Lemma \ref{l3.1}, we have $J(t)\le C$. Similar to (\ref{3.10}), we have 
$$\rho^*(x,t)=\frac{1}{u(x,t)}.$$
Let $s=\frac{1}{r}$, we have
\begin{equation}\label{3.11}
\begin{split}
C\ge J(t)=&\int_{\mS^n}\int_\frac{1}{C_1}^\frac{1}{\rho^*(x,t)}\varphi^{-\frac{n}{\beta}}(x,s)dsdx\\
=&\int_{\mS^n}\int_{\rho^*(x,t)}^{C_1}\varphi^{-\frac{n}{\beta}}(x,\frac{1}{r})r^{-2}drdx\\
=&\int_{\mS^n}\int_{\rho^*(\xi^*,t)}^{C_1}\varphi^{-\frac{n}{\beta}}(\xi^*,\frac{1}{r})r^{-2}drd\xi^*.
\end{split}
\end{equation}
Define $\widetilde G(\xi^*,r)=\varphi^{-\frac{n}{\beta}}(\xi^*,\frac{1}{r})r^{-n-2}$. By $\rho^*\le C_1$, (\ref{3.11}) implies
\begin{equation}\label{3.12}
\int_{B_{C_1}\backslash \Om_t^*}\widetilde G(y)dy=\int_{\mS^n}\int_{\rho^*(\xi^*,t)}^{C_1}\widetilde G(\xi^*,r)r^ndrd\xi^*\le C.
\end{equation}
Similarly,
\begin{equation*}
\begin{split}
+\infty=\int_1^\infty\int_{\mS^n}\varphi^{-\frac{n}{\beta}}(x,s)dxds
=&\int_0^1\int_{\mS^n}\varphi^{-\frac{n}{\beta}}(x,\frac{1}{r})r^{-2}dxdr\\
=&\int_0^1\int_{\mS^n}\varphi^{-\frac{n}{\beta}}(\xi^*,\frac{1}{r})r^{-2}d\xi^*dr\\
=&\int_0^1\int_{\mS^n}\widetilde G(\xi^*,r)r^nd\xi^*dr.
\end{split}
\end{equation*}
For fixed positive constant $C_1$, the above equation means
\begin{equation}\label{3.13}
\int_0^{C_1}\int_{\mS^n}\widetilde G(\xi^*,r)r^nd\xi^*dr=+\infty.
\end{equation}
Recalling the proof of the case (\rmnum{3}) of Lemma \ref{l3.2}, we can find that (\ref{3.12}), (\ref{3.13}) and  (\ref{3.8}), (\ref{3.6}) have the same meaning respectively. By the proof of the case (\rmnum{3}) of Lemma \ref{l3.2}, we derive the lower bound of $u^*$, i.e. we derive the upper bound of $u$.

(2) If $\int_1^\infty\int_{\mS^n}\varphi(x,s)^{-\frac{n}{\beta}}dxds<+\infty$.  Let 
$$J(t)=\int_{\mS^n}\int_u^\infty\varphi(x,s)^{-\frac{n}{\beta}}dsdx.$$
By the calculation of Lemma \ref{l3.1}, $J(t)$ is non-decreasing, i.e. $J(t)\ge J(0)$.

Denote $\phi(x,u)=\int_u^\infty\varphi(x,s)^{-\frac{n}{\beta}}ds$. Let $\delta$ be a small positive constant satisfying the following inequality:
\begin{equation}\label{3.17}
\vert\{x\in\mS^n:|x_{n+1}|<\delta\}|\le\frac{J(0)}{2\max_{\mS^n}\phi(x,\frac{1}{C_1})}.
\end{equation}
Due to convexity and Lemma \ref{l2.1},
$$u\ge M|x_{n+1}|.$$
Denoting $S_\delta=\{x\in\mS^n:|x_{n+1}|<\delta\}$ and noting that $\phi(x,u)$ is decreasing with $u$, we have
\begin{equation}\label{3.18}
	\begin{split}
J(0)\le& J(t)\\
=&\int_{\mS^n\backslash S_\delta}\phi(x,u)dx+\int_{S_\delta}\phi(x,u)dx\\
\le&\int_{\mS^n\backslash S_\delta}\phi(x,M|x_{n+1}|)dx+\int_{S_\delta}\phi(x,\frac{1}{C_1})dx\\
\le&\int_{\mS^n\backslash S_\delta}\phi(x,M\delta)dx+\max_{x\in\mS^n}\phi(x,\frac{1}{C_1})|S_\delta|\\
\le&\max_{x\in\mS^n}\phi(x,M\delta)|S^n|+\max_{x\in\mS^n}\phi(x,\frac{1}{C_1})|S_\delta|.
	\end{split}
\end{equation}
Inserting (\ref{3.17}) into (\ref{3.18}), we obtain
\begin{align*}
\frac{J(0)}{2}\le \max_{x\in\mS^n}\phi(x,M\delta)|S^n|,
\end{align*}
which implies that $\max_{x\in\mS^n}\phi(x,M\delta)$ has a uniformly positive lower bound. By its definition, $\phi(x,s)$ is decreasing and tends to $0^+$ as $s\rightarrow+\infty$. Thus $M$ is uniformly bounded from above. Thus we derive the upper bound of $u$.
\end{proof}
\textbf{Remark:} Case (2) in this proof can be also proved by the dual relation. If $\int_{B_1}G(y)dy<\infty$, the dual body satisfies $\int_1^\infty\int_{\mS^n}\widetilde\varphi(x,s)^{-\frac{n}{\beta}}dxds<+\infty$, where the definition of $\widetilde\varphi(x,s)$ can be seen in the proof of case (\rmnum{6}) in Lemma \ref{l3.6}. Thus by the dual relation and the proof of Lemma \ref{l3.5} (2), we can prove case (2).

Next we give the $C^0$ estimates for cases (\rmnum{1}) and (\rmnum{2}) of Theorem \ref{t1.2}.
\begin{lemma}\label{l3.4}
	Under flow (\ref{1.1}), the corresponding assumptions of Theorem \ref{t1.2} and the cases (\rmnum{1}) and (\rmnum{2}), there exists a positive constant $C_3$ independent of $t$, such that
	$$u(x,t)\le  C_{3}.$$
It also means that
$$\rho,|Du|\le C_{3}$$
by Lemma \ref{l2.1}.
\end{lemma}
\begin{proof}
We use the method in the proof of Lemma \ref{l3.3}. We can find that the cases (\rmnum{1}) and (\rmnum{2}) correspond to the cases (\rmnum{3}) and (\rmnum{4}) respectively. We only prove the case (\rmnum{1}) since the proof of case (\rmnum{2}) is similar to case (\rmnum{1}). Let 
$$J(t)=\int_{\mS^n}\int_1^u\varphi(x,s)^{-\frac{n}{\beta}}dsdx.$$
By the calculation of Lemma \ref{l3.1}, $J(t)$ is non-increasing, i.e. $J(t)\le J(0)$. Similar to above, by the conditions of case (\rmnum{1}), we have
\begin{align*}
\int_{\mS^n}\int^{1}_{\rho^*}\widetilde G(\xi^*,r)r^ndrd\xi^*=\int_{\mS^n}\int^{u}_{1}\varphi^{-\frac{n}{\beta}}(x,s)dsdx\le J(0),\\
	\int_{B_1}\widetilde G(y)dy=\int_{0}^1\int_{\mS^n}\widetilde G(\xi^*,r)r^nd\xi^*dr=\int_{1}^\infty\int_{\mS^n}\varphi^{-\frac{n}{\beta}}(x,s)dxds=+\infty,\\
\int_{\mS^n}\int^\frac{1}{|\la\xi^*,\theta\ra|}_1\widetilde G(\xi^*,r)r^ndrd\xi^*=\int_{\mS^n}\int_{|\la x,\theta\ra|}^1\varphi^{-\frac{n}{\beta}}(x,s)dsdx\le\hat{C}_1,
\end{align*}
where $\widetilde G(\xi^*,r)=\widetilde G(x,r)=\varphi^{-\frac{n}{\beta}}(x,\frac{1}{r})r^{-n-2}$ and $r=\frac{1}{s}$. Repeat the proof of case (\rmnum{3}) in Lemma \ref{l3.2} to the dual body $\Om_t^*$. Then we derive the lower bound of $\rho^*$. Thus we derive the upper bound of $u$.
\end{proof}
We can also derive the lower bound of $u$ along flow (\ref{1.1}) if we have the upper bound of $u$.
\begin{lemma}\label{l3.5}
Let $M_0$ be a closed, smooth, origin-symmetric, uniformly convex hypersurface in $\mathbb{R}^{n+1}$, $n\ge2$, enclosing the origin. Suppose $\varphi:\mS^n\times(0,+\infty)\rightarrow(0,+\infty)$ and $G:\mR^{n+1}\backslash\{0\}\rightarrow(0,+\infty)$ are two smooth and even functions. If $u(\cdot,t)\le C_3,$  along flow (\ref{1.1}),
there exists a positive constant $C_4$ independent of $t$, such that
$$u(x,t)\ge  C_{4}.$$
It also means that
$$\rho\ge  C_{4}$$
by Lemma \ref{l2.1}.
\end{lemma}
\begin{proof}
We divide this proof into two cases.

(1) If $\int_{B_1}G(y)^\frac{n}{\beta}dy=\infty$. Recall the proof of case (\rmnum{3}) in Lemma \ref{l3.2}. At this time we let
$$V(t)=\int_{B_{C_3}\backslash \Om_t}G(y)^{\frac{n}{\beta}}dy,$$
which is only used in this case.
Since $u\le C_3,$ we have $\Om_t\subset B_{C_3}$, where $B_{C_3}$ denotes the ball centered at the origin with radius $C_3$. The calculation in Lemma \ref{l3.1} can be used to this situation, i.e.
$$\int_{B_{C_3}\backslash \Om_t}G(y)^{\frac{n}{\beta}}dy=\int_{B_{C_3}\backslash \Om_0}G(y)^{\frac{n}{\beta}}dy=C.$$
 Replace $B_1$ in the proof of case (\rmnum{3}) in Lemma \ref{l3.2} with $B_{C_3}$ then we can get the lower bound of $u$ by (\ref{3.6}).

(2) If $\int_{B_1}G(y)^\frac{n}{\beta}dy<\infty$. At this time we let
$$V(t)=\int_{\Om_t}G(y)^{\frac{n}{\beta}}dy.$$ 
Since $u\le C_3$, there is $\Om_t\subset B_{C_3}.$
 $\int_{B_1}G(y)dy<\infty$ implies that there exists a $\delta>0$ such that
$$\int_AG(y)dy<V(0) \text{ for every measurable set } A\subset B_{C_3} \text{ with }|A|<\delta.$$
By Lemma \ref{l3.1}, we have $\int_{\Om_t}G(y)dy=V(0)$, i.e. Vol$(\Om_t)\ge\delta$. By rotation, we can assume $\rho(e_{1},t)=\rho_{\min}(t)$. Since $\Om_t$ is origin-symmetric, we find that $\Om_t$ is contained in the cube
$$Q_t=\{z\in\mR^{n+1}:-\rho_{\max}\le z_i\le\rho_{\max} \text{ for }2\le i\le n+1,-\rho_{\min}\le z_{1}\le\rho_{\min}\}.$$
Therefore,
\begin{equation}\label{3.16}
	\delta\le\text{Vol}(\Om_t)\le 2^{n+1}(\rho_{\max})^n\rho_{\min}.
\end{equation}
By $\rho\le C_3$ we derive the lower bound of $\rho$.

In  summary we have completed this proof.
\end{proof}
\textbf{Remark:} Case (2) in this proof can be also proved by the dual relation. If $\int_{B_1}G(y)^\frac{n}{\beta}dy<\infty$, the dual body satisfies $\int_1^\infty\int_{\mS^n}\widetilde\varphi(x,s)^{-\frac{n}{\beta}}dxds<+\infty$. Thus by the dual relation and the proof of Lemma \ref{l3.3} (2), we can prove case (2).

 We next show the $C^0$ estimates with the cases (\rmnum{5}) and (\rmnum{6}) of Theorem \ref{t1.2}.
\begin{lemma}\label{l3.6}
	Under flow (\ref{1.1}), the corresponding assumptions of Theorem \ref{t1.2} and the cases (\rmnum{5}) and (\rmnum{6}), there exists a positive constant $C_5$ independent of $t$, such that
	$$\frac{1}{C_5}\le u(x,t)\le  C_{5}.$$
	It also means that
	$$\frac{1}{C_5}\le \rho\le C_{5},\quad |Du|\le C_{5}$$
	by Lemma \ref{l2.1}.
\end{lemma}
\begin{proof}
Case (\rmnum{5}):
This proof is inspired by \cite{DL3}. First, we shall derive the upper bound of $u$. We follow an argument in \cite{CW}. Since $\lim_{\mu\rightarrow\infty}\int_{\mu}^\infty\int_{\mS^n}\varphi^{-\frac{n}{\beta}}(x,s)dxds=0$ by $\int_{\mR^{n+1}\backslash B_1}\varphi^{-\frac{n}{\beta}}(y)|y|^{-n}dy<+\infty$, We can choose $C_6$ such that
\begin{equation}\label{3.21}
\int_{\mS^n}\int_{C_6}^\infty\varphi^{-\frac{n}{\beta}}(x,s)dsdx<\frac{J(0)}{2},
\end{equation} 
where we let $J(t)=\int_{\mS^n}\int_{u}^\infty\varphi^{-\frac{n}{\beta}}(x,s)dsdx.$ Note that 
$${\lim\sup}_{s\rightarrow0^+}\frac{\int_{s}^\infty\varphi^{-\frac{n}{\beta}}(x,t)dt}{s^p}<\infty.$$
In other words, there exists $\delta>0$ such that
\begin{equation}\label{3.22}
\int_{s}^\infty\varphi^{-\frac{n}{\beta}}(x,t)dt\le C_7s^p
\end{equation}
for $\forall s<\delta$. Set $S_1(t)=\mS^n\cap\{u(t)<\delta\}$, $S_2(t)=\mS^n\cap\{\delta\le u(t)<C_6\}$ and $S_3(t)=\mS^n\cap\{u(t)\ge C_6\}$.
Suppose there is a sequence origin-symmetric convex body $\Om_{t_j}$ along this flow, and the diameter of $\Om_{t_j}$, $2M(t_j)\rightarrow\infty$ as $t_j\rightarrow T$. Since $\Om_{t_j}$ is origin-symmetric, we have $u_j(y)\ge M(t_j)|x_0\cdot y|$ for any $y\in\mS^n$, where $u_j$ attains the maximum at $x_0\in\mS^n$. We conclude that
\begin{equation}\label{3.23}
	\vert S_1(t_j)\vert\rightarrow0,\quad \vert S_2(t_j)\vert\rightarrow0\quad \text{ as } M(t_j)=\max_{\mS^n}u(t_j)\rightarrow\infty.
\end{equation}
Let $\Om^*$ be the polar set of $\Om$, and $\rho^*(t)=\rho_{\Om_{t}^*}$. Recall $-p<q^*$. Thus by Lemma \ref{l3.1}, (\ref{3.21}) and (\ref{3.22}),
\begin{equation}\label{3.24}
	\begin{split}
		J(0)\le&\int_{S_1(t)\cup S_2(t)\cup S_3(t)}\int_{u}^\infty\varphi^{-\frac{n}{\beta}}(x,s)dsdx\\
	\le&C_7\int_{S_1(t)}(\rho^*)^{-p}dx+C_8\vert S_2\vert+\int_{S_3}\int_{C_6}^\infty\varphi^{-\frac{n}{\beta}}(x,s)dsdx\\
\le&C_7\(\int_{S_1(t)}(\rho^*)^{q'}dx\)^{-\frac{p}{q'}}|S_1|^{1+\frac{p}{q'}}+C_8\vert S_2\vert+\int_{\mS^n}\int_{C_6}^\infty\varphi^{-\frac{n}{\beta}}(x,s)dsdx,
	\end{split}
\end{equation}
for any $-p<q'<q^*$. Since by Lemma \ref{3.1}, we have 
\begin{equation}\label{3.25}
V(0)=\int_{\Om_t} G(r\xi)^\frac{n}{\beta}(y)dy\le \hat{C}_5\int_{\mS^n}\rho^q(\xi)d\xi.
\end{equation}
By Lemma \ref{l2.2} and (\ref{3.25}), we have
$$\(\int_{S_1(t)}(\rho^*)^{q'}dx\)^{-\frac{p}{q'}}\le C_{9}.$$
If $M(t_j)\rightarrow\infty$, then by (\ref{3.21}), (\ref{3.23}) and (\ref{3.24}),
\begin{equation*}
J(0)\le \frac{J(0)}{2}.
\end{equation*}
Thus we arrive a contradiction, i.e. we derive the upper bound of $u$. By Lemma \ref{l3.5} we get the lower bound of $u$. 



Case (\rmnum{6}): We use the method in the proof of Lemma \ref{l3.3}. Let $s=\frac{1}{r}$,
$$V(0)=V(t)=\int_{\mS^n}\int_{0}^\rho G^{\frac{n}{\beta}}(\xi,r)r^ndrd\xi =\int_{\mS^n}\int_{u^*}^\infty\widetilde\varphi^{-\frac{n}{\beta}}(\xi,s)dsd\xi,$$
where $\widetilde\varphi^{-\frac{n}{\beta}}(\xi,s)=G^{\frac{n}{\beta}}(\xi,\frac{1}{s})s^{-n-2}$. We also have
$${\lim\sup}_{\eps\rightarrow0^+}\frac{\int_{\eps}^\infty\widetilde\varphi^{-\frac{n}{\beta}}(x,s)ds}{\eps^p}={\lim\sup}_{\eps\rightarrow0^+}\frac{\int_{0}^\frac{1}{\eps}G(r\xi)^\frac{n}{\beta}r^ndr}{\eps^p}<\infty,$$
$$J(0)\le\int_{\mS^n}\int_{u(x)}^\infty\varphi^{-\frac{n}{\beta}}(x,s)dsdx\le \hat{C}_6\int_{\mS^n}u^{-q}(x)dx=\hat{C}_6\int_{\mS^n}(\rho^*)^{q}d\xi^*.$$
Repeat the proof of case (\rmnum{5}) for the dual body $\Om_{t}^*$, then we derive the bound of $u^*$. In other words, we derive the bound of $u$.

In summary, we have completed this proof.
\end{proof}
Next we show the $C^0$ estimates of Theorem \ref{t1.1}.
\begin{lemma}\label{l3.7}
	Under flow (\ref{1.1}) and the corresponding assumptions of Theorem \ref{t1.1}, there exists a positive constant $C_{10}$ independent of $t$, such that
	$$\frac{1}{C_{10}}\le u(x,t)\le  C_{10}.$$
	It also means that
	$$\frac{1}{C_{10}}\le \rho\le C_{10},\quad |Du|\le C_{10}$$
	by Lemma \ref{l2.1}.
\end{lemma}
\begin{proof}
Case (\rmnum{2}a): By $\int_{0}^1\int_{\mS^n}\varphi^{-\frac{n}{\beta}}(x,s)dxds<\infty$, we can define 
$$J(t)=\int_{\mS^n}\int_{0}^u\varphi^{-\frac{n}{\beta}}(x,s)dsdx<\infty.$$
 By the calculation of Lemma \ref{3.1}, we derive that $J(t)$ is non-increasing, i.e. $J(t)\le J(0)$. Due to convexity and Lemma \ref{l2.1}, 
$$u\ge Mx_{n+1}\quad \text{ on } \{x_{n+1}>0\}\cap\mS^n.$$
 Then
\begin{equation}\label{3.29}
\begin{split}
J(0)\ge& J(t)\\
\ge&\int_{\{x_{n+1}>\eps\}\cap\mS^n}\int_{0}^u\varphi^{-\frac{n}{\beta}}(x,s)dsdx\\
\ge&\int_{\{x_{n+1}>\eps\}\cap\mS^n}\int_{0}^{M\eps}\varphi^{-\frac{n}{\beta}}(x,s)dsdx.
\end{split}
\end{equation}
Due to $\int_{1}^\infty\int_{\{\la x,\theta_1\ra\ge\eps\}\cap\mS^n}\varphi^{-\frac{n}{\beta}}(x,s)dx ds=+\infty$, we see that $u\le C_{10}$ is independent with $t$. Then we can define
$$V(t)=\int_{B_{C_{10}}\backslash \Om_t}G(y)^{\frac{n}{\beta}}dy=\int_{\mS^n}\int_\rho^{C_{10}}G(r\xi)^{\frac{n}{\beta}}r^ndrd\xi.$$
Due to convexity and Lemma \ref{l2.1}, 
$$\rho\le \frac{m}{\la\xi,e_1\ra}\quad \text{ on } \{\la\xi,e_1\ra>0\}\cap\mS^n.$$
Thus by Lemma \ref{l3.1}
\begin{equation}\label{3.30}
	\begin{split}
	V(0)\ge& V(t)\\
		\ge&\int_{\{\la\xi,e_1\ra>\delta\}\cap\mS^n}\int_\rho^{C_{10}}G(r\xi)^{\frac{n}{\beta}}r^ndrd\xi\\
		\ge&\int_{\{\la\xi,e_1\ra>\delta\}\cap\mS^n}\int_{\frac{m}{\delta}}^{C_{10}}G(r\xi)^{\frac{n}{\beta}}r^ndrd\xi\\
	=&\int_{\frac{m}{\delta}}^{C_{10}}\int_{\{\la\xi,e_1\ra>\delta\}\cap\mS^n}G(r\xi)^{\frac{n}{\beta}}r^nd\xi dr.
	\end{split}
\end{equation}
Due to $\int_{0}^1\int_{\{\la\xi,\theta_2\ra\ge\delta\}\cap\mS^n}G(r\xi)^\frac{n}{\beta}r^nd\xi dr=+\infty$, we derive the lower bound of $u$.

Case (\rmnum{2}b): At this case we can define
$$V(t)=\int_{\mR^{n+1}\backslash \Om_t}G(y)^{\frac{n}{\beta}}dy=\int_{\mS^n}\int_\rho^{\infty}G(r\xi)^{\frac{n}{\beta}}r^ndrd\xi.$$
Similar to (\ref{3.30}) we derive the lower bound of $u$, i.e. $u\ge\frac{1}{C_{10}}$. we can define 
$$J(t)=\int_{\mS^n}\int_{\frac{1}{C_{10}}}^u\varphi^{-\frac{n}{\beta}}(x,s)dsdx.$$
Then similar to (\ref{3.29}) we can derive the upper bound of $u$.
\end{proof}
Finally we show the $C^0$ estimates of Theorem \ref{xt1.3} and close this section.
\begin{lemma}\label{l3.8}
	Under flow (\ref{x1.7}) and the corresponding assumptions of Theorem \ref{xt1.3}, there exists a positive constant $C_{11}$ independent of $t$, such that
	$$\frac{1}{C_{11}}\le u(x,t)\le  C_{11}.$$
	It also means that
	$$\frac{1}{C_{11}}\le \rho\le C_{11},\quad |Du|\le C_{11}$$
	by Lemma \ref{l2.1}.
\end{lemma}
\begin{proof}
This proof is inspired by \cite{DL4} Lemma 4.1. Let $u_{\max}(t)=\max_{x\in \mS^n}u(\cdot,t)=u(x_t,t)$. For fixed time $t$, at the point $x_t$, we have $$D_iu=0 \text{ and } D^2_{ij}u\leq0.$$
Note that $h_{ij}=D^2_{ij}u+u\delta_{ij}$.
At the point $x_t$, we have $\sigma_n^\frac{\beta}{n}(h_{ij})\leq u^{\beta}$. Then
$$\frac{d}{dt}u_{\max}\leq u(\varphi G u^\beta-1).$$
Note that $\rho(x_t)=\max_{\mS^n}\rho=\max_{\mS^n}u$. If we denote $X=\rho\xi$, then $u(x_t)=\la X,x\ra=\rho(x_t)\la x,\xi\ra$, i.e. $\la x,\xi\ra|_{x_t}=1$ and $X=u_{x_t}x$.
Since $r_1<\rho_{M_0}<r_2$, we can consider the point where the maximum point touches the sphere $|X|=r_2$ for the first time. By (\ref{1.8}) we have $\p_tu\le0$ at maximum points. Hence
$$u_{\max}\leq \max\{r_2,u_{\max}(0)\}.$$
Similarly, at the minimum point, 
$$\frac{d}{dt}u_{\min}\geq u(\varphi Gu^\beta-1).$$
 Hence by the maximum principle,
$$u_{\min}\geq \min\{r_1,u_{\min}(0)\}.$$
\end{proof}

\section{Convergence and proof of theorems}
\begin{proofs of Theorem 1.1-1.5}
We only need to prove Theorem \ref{t1.1}, \ref{t1.2} and \ref{xt1.3} since Theorem \ref{t1.3} and \ref{t1.4} are direct corollaries of Theorem \ref{t1.1}, \ref{t1.2} and \ref{xt1.3}.
We have  proved the $C^0$ and $C^1$ estimates in Section 3. In detail,
we have proved the $C^0$ and $C^1$ estimates of cases (\rmnum{3}) and (\rmnum{4}) of Theorem \ref{t1.2} in Lemma \ref{l3.2} and \ref{l3.3}; the $C^0$ and $C^1$ estimates of cases (\rmnum{1}) and (\rmnum{2}) of Theorem \ref{t1.2} have been proved in Lemma \ref{l3.4} and \ref{l3.5}; the $C^0$ and $C^1$ estimates of cases (\rmnum{5}) and (\rmnum{6}) of Theorem \ref{t1.2} have been proved in Lemma \ref{l3.6}; the $C^0$ and $C^1$ estimates of Theorem \ref{t1.1} and \ref{xt1.3} have been proved in Lemma \ref{l3.7} and \ref{l3.8} respectively.

Next we want to derive the bounds of $\eta(t)$. This is obvious in Theorem \ref{xt1.3}. Then we shall derive the bounds of $\eta(t)$ under the assumptions of Theorem \ref{t1.1} or \ref{t1.2}.
By $C^0$ estimates,
$$\frac{1}{C_{12}}\le\int_{\mS^n}G(\xi,\rho)^\frac{n}{\beta}\rho^{n+1}d\xi\le C_{12}.$$
By (\ref{2.6}) and $C^0$ estimates, we have $\frac{1}{C_{13}}\le\int_{\mS^n}\sigma_{n}dx\le C_{13}.$ This means
$$\int_{\mS^n} u\varphi(x,u)G^{\frac{n}{\beta}+1}(X)\sigma_{n}^{\frac{\beta}{n}+1}dx\ge C_{14}\int_{\mS^n}\sigma_{n}^{\frac{\beta}{n}+1}dx\ge C_{15}(\int_{\mS^n}\sigma_{n}dx)^{\frac{\beta}{n}+1}\ge C_{16},$$
i.e. we have the upper bound of $\eta(t)$. To derive the lower bound of $\eta(t)$,
by \cite{DL4} Lemma 3.3, we only need to show 
\begin{equation}\label{4.1}
\eta(t)\max \sigma_{n}^{\frac{\beta}{n}}\ge C
\end{equation}
to derive the upper bound of $\sigma_{n}$. In fact, by (\ref{2.6}) and the $C^0$ estimates, we have
\begin{equation*}
\eta(t)\max \sigma_{n}^{\frac{\beta}{n}}\ge C_{17}\max \sigma_{n}^{\frac{\beta}{n}}\dfrac{\int_{\mS^n}\sigma_{n}dx}{\int_{\mS^n}\sigma_{n}^{\frac{\beta}{n}+1}dx}\ge C_{17}.
\end{equation*}
By \cite{DL4} Section 3 we derive the $C^1$ and $C^2$ estimates, see also \cite{CL}. 
The $C^0$, $C^1$ and $C^2$ estimates are independent of $T^*$. By $C^0$, $C^1$ and $C^2$ estimates, we conclude that the equation (\ref{2.2}) is uniformly parabolic. By the $C^0$ estimate, the gradient estimate, the $C^2$ estimate, Cordes and Nirenberg type estimates \cite{B2,CO,LN} and the Krylov's theory \cite{KNV}, we get the H$\ddot{o}$lder continuity of $D^2u$ and $u_t$. Then we can get higher order derivation estimates by the regularity theory of the uniformly parabolic equations. Hence we obtain the long time existence and $C^\infty$-smoothness of solutions for the normalized flow (\ref{2.2}). The uniqueness of smooth solutions also follows from the parabolic theory.

Next we will prove that the support function $u_\infty$ of $M_\infty$ satisfies the following equation
$$\varphi(x,u)G(X)\sigma_{n}^\frac{\beta}{n}=c.$$
The convergence of flow (\ref{x1.7}) had been proved in proof of Theorem 1.4 in \cite{DL4}. Finally we will prove the convergence of flow (\ref{1.1}).
By the monotonicity of $J(t)$ we have
$$\vert\int_{0}^t\frac{dJ(s)}{ds}ds\vert=\vert J(t)-J(0)\vert<\infty.$$
Then $\frac{dJ(t)}{dt}=0$ if and only if $M_t$ satisfies $\varphi(x,u)G(X)\sigma_{n}^\frac{\beta}{n}=c.$ 
Hence there is a sequence of $t_i\rightarrow\infty$ such that $\frac{d}{dt}V\rightarrow0$, i.e. $u(\cdot,t_i)$ converges smoothly to a positive, smooth and uniformly convex function $u_\infty$ solving $\varphi(x,u)G(X)\sigma_{n}^\frac{\beta}{n}=c.$
\end{proofs of Theorem 1.1-1.5}

\section{Reference}



\begin{biblist}





\bib{B2}{article}{
	author={Andrews B.},
	author={ McCoy J.},
	title={Convex hypersurfaces with pinched principal curvatures and flow of convex hypersurfaces by high powers of curvature},
	journal={Trans. Amer. Math. Soc.},
	volume={364(7)}
	date={2012},
	pages={3427-3447},
}



\bib{BBC}{article}{
	author={Bianchi G.},
	author={B\"{o}r\"{o}czky K.J.},
	author={Colesanti A.},
	title={The Orlicz version of the $L_p$ Minkowski problem for
		$-n<p<0$},
	journal={Adv. in Appl. Math.},
	volume={111},
	date={2019},
	pages={101937, 29},
	issn={0196-8858},
	review={\MR{3998833}},
	doi={10.1016/j.aam.2019.101937},
}



\bib{BF}{article}{
	author={B\"{o}r\"{o}czky K. J.},
	author={Fodor F.},
	title={The $L_p$ dual Minkowski problem for $p>1$ and $q>0$},
	journal={J. Differential Equations},
	volume={266},
	date={2019},
	number={12},
	pages={7980--8033},
	issn={0022-0396},
	review={\MR{3944247}},
	doi={10.1016/j.jde.2018.12.020},
}


\bib{BLY}{article}{
	author={B\"{o}r\"{o}czky K. J.},
	author={Lutwak E.},
	author={Yang D.},
	author={Zhang G.},
	author={Zhao Y.},
	title={The dual Minkowski problem for symmetric convex bodies},
	journal={Adv. Math.},
	volume={356},
	date={2019},
	pages={106805, 30},
	issn={0001-8708},
	review={\MR{4008522}},
	doi={10.1016/j.aim.2019.106805},
}


\bib{BIS}{article}{
	author={Bryan P.},
	author={Ivaki M. N.},
	author={Scheuer J.},
	title={A unified flow approach to smooth, even $L_p$-Minkowski problems},
	journal={Anal. PDE},
	volume={12},
	date={2019},
	number={2},
	pages={259--280},
	issn={2157-5045},
	review={\MR{3861892}},
	doi={10.2140/apde.2019.12.259},
}


\bib{BIS2}{article}{
	author={Bryan P.},
	author={Ivaki M. N.},
	author={Scheuer J.},
	title={Parabolic approaches to curvature equations},
	journal={Nonlinear Anal.},
	volume={203},
	date={2021},
	pages={Paper No. 112174, 24},
	issn={0362-546X},
	review={\MR{4172901}},
	doi={10.1016/j.na.2020.112174},
}

\bib{BIS3}{article}{
	author={Bryan P.},
author={Ivaki M. N.},
author={Scheuer J.},
	title={Orlicz-Minkowski flows},
	journal={Calc. Var. Partial Differential Equations},
	volume={60},
	date={2021},
	number={1},
	pages={Paper No. 41, 25},
	issn={0944-2669},
	review={\MR{4204567}},
	doi={10.1007/s00526-020-01886-3},
}


\bib{CO}{article}{
	author={Cordes H. O.},
	title={$\ddot U$ber die erste Randwertaufgabe bei quasilinearen Differentialgleichungen zweiter Ordnung in mehr als zwei Variablen},
	journal={Math. Ann.},
	volume={131}
	date={1956},
	pages={278-312},
}




\bib{CHZ}{article}{
	author={Chen C.},
	author={Huang Y.},
	author={Zhao Y.},
	title={Smooth solutions to the $L_p$ dual Minkowski problem},
	journal={Math. Ann.},
	volume={373},
	date={2019},
	number={3-4},
	pages={953--976},
	issn={0025-5831},
	review={\MR{3953117}},
	doi={10.1007/s00208-018-1727-3},
}


\bib{CHD}{article}{
	author={Chen H.},
	title={On a generalised Blaschke–Santalò inequality},
	date={2018},
	number={preprint},
	journal={arXiv:1808.02218},
}


\bib{CCL}{article}{
	author={Chen H.},
	author={Chen S.},
	author={Li Q.},
	title={Variations of a class of Monge-Amp\`ere-type functionals and their
		applications},
	journal={Anal. PDE},
	volume={14},
	date={2021},
	number={3},
	pages={689--716},
	issn={2157-5045},
	review={\MR{4259871}},
	doi={10.2140/apde.2021.14.689},
}



\bib{CL}{article}{
	author={Chen H.},
	author={Li Q.},
	title={The $L_ p$ dual Minkowski problem and related parabolic flows},
	journal={J. Funct. Anal.},
	volume={281},
	date={2021},
	number={8},
	pages={Paper No. 109139, 65},
	issn={0022-1236},
	review={\MR{4271790}},
	doi={10.1016/j.jfa.2021.109139},
}

\bib{CLL}{article}{
	author={Chen L.},
	author={Liu Y.},
	author={Lu J.},
	author={Xiang N.},
	title={Existence of smooth even solutions to the dual Orlicz-Minkowski
		problem},
	journal={J. Geom. Anal.},
	volume={32},
	date={2022},
	number={2},
	pages={Paper No. 40, 25},
	issn={1050-6926},
	review={\MR{4358692}},
	doi={10.1007/s12220-021-00803-0},
}




\bib{CW}{article}{
	author={Chou K.},
	author={Wang X.},
	title={The $L_p$-Minkowski problem and the Minkowski problem in
		centroaffine geometry},
	journal={Adv. Math.},
	volume={205},
	date={2006},
	number={1},
	pages={33--83},
	issn={0001-8708},
	review={\MR{2254308}},
	doi={10.1016/j.aim.2005.07.004},
}


\bib{DL}{article}{
	author={Ding S.},
	author={Li G.},
	title={A class of curvature flows expanded by support function and
		curvature function},
	journal={Proc. Amer. Math. Soc.},
	volume={148},
	date={2020},
	number={12},
	pages={5331--5341},
	issn={0002-9939},
	review={\MR{4163845}},
	doi={10.1090/proc/15189},
}

\bib{DL3}{article}{
	author={Ding S.},
	author={Li G.},
	title={A class of inverse curvature flows and $L_p$ dual Christoffel-Minkowski problem},
	pages={	to appear in Transactions of the American Mathematical Society},
}

\bib{DL4}{article}{
	author={Ding S.},
	author={Li G.},
	title={Anisotropic flows without global terms and dual Orlicz Christoffel-Minkowski type problem},
	pages={	arXiv:2207.03114},
}

\bib{FH}{article}{
	author={Feng Y.},
	author={He B.},
	title={The Orlicz Aleksandrov problem for Orlicz integral curvature},
	journal={Int. Math. Res. Not. IMRN},
	date={2021},
	number={7},
	pages={5492--5519},
	issn={1073-7928},
	review={\MR{4241134}},
	doi={10.1093/imrn/rnz384},
}

\bib{FHL}{article}{
	author={Feng Y.},
	author={Hu S.},
	author={Liu W.},
	title={Existence and uniqueness of solutions to the Orlicz Aleksandrov
		problem},
	journal={Calc. Var. Partial Differential Equations},
	volume={61},
	date={2022},
	number={4},
	pages={Paper No. 148, 23},
	issn={0944-2669},
	review={\MR{4434164}},
	doi={10.1007/s00526-022-02260-1},
}


\bib{FWJ2}{article}{
	author={Firey W. J.},
	title={The determination of convex bodies from their mean radius of
		curvature functions},
	journal={Mathematika},
	volume={14},
	date={1967},
	pages={1--13},
	issn={0025-5793},
	review={\MR{217699}},
	doi={10.1112/S0025579300007956},
}

\bib{GHW1}{article}{
	author={Gardner J.},
	author={Hug D.},
	author={Weil W.},
	author={Xing S.},
	author={Ye D.},
	title={General volumes in the Orlicz-Brunn-Minkowski theory and a related
		Minkowski problem I},
	journal={Calc. Var. Partial Differential Equations},
	volume={58},
	date={2019},
	number={1},
	pages={Paper No. 12, 35},
	issn={0944-2669},
	review={\MR{3882970}},
	doi={10.1007/s00526-018-1449-0},
}

\bib{GHW2}{article}{
	author={Gardner J.},
author={Hug D.},
author={Xing S.},
author={Ye D.},
	title={General volumes in the Orlicz-Brunn-Minkowski theory and a related
		Minkowski problem II},
	journal={Calc. Var. Partial Differential Equations},
	volume={59},
	date={2020},
	number={1},
	pages={Paper No. 15, 33},
	issn={0944-2669},
	review={\MR{4040624}},
	doi={10.1007/s00526-019-1657-2},
}




\bib{GM}{article}{
	author={Guan P.},
	author={Ma X.},
	title={The Christoffel-Minkowski problem. I. Convexity of solutions of a
		Hessian equation},
	journal={Invent. Math.},
	volume={151},
	date={2003},
	number={3},
	pages={553--577},
	issn={0020-9910},
	review={\MR{1961338}},
	doi={10.1007/s00222-002-0259-2},
}



\bib{HLYZ}{article}{
	author={Haberl C.},
	author={Lutwak E.},
	author={Yang D.},
	author={Zhang G.},
	title={The even Orlicz Minkowski problem},
	journal={Adv. Math.},
	volume={224},
	date={2010},
	number={6},
	pages={2485--2510},
	issn={0001-8708},
	review={\MR{2652213}},
	doi={10.1016/j.aim.2010.02.006},
}







\bib{HLY2}{article}{
	author={Huang Y.},
	author={Lutwak E.},
	author={Yang D.},
	author={Zhang G.},
	title={Geometric measures in the dual Brunn-Minkowski theory and their
		associated Minkowski problems},
	journal={Acta Math.},
	volume={216},
	date={2016},
	number={2},
	pages={325--388},
	issn={0001-5962},
	review={\MR{3573332}},
	doi={10.1007/s11511-016-0140-6},
}

\bib{HLY}{article}{
	author={Huang Y.},
author={Lutwak E.},
author={Yang D.},
author={Zhang G.},
	title={The $L_p$-Aleksandrov problem for $L_p$-integral curvature},
	journal={J. Differential Geom.},
	volume={110},
	date={2018},
	number={1},
	pages={1--29},
	issn={0022-040X},
	review={\MR{3851743}},
	doi={10.4310/jdg/1536285625},
}



\bib{HZ}{article}{
	author={Huang Y.},
	author={Zhao Y.},
	title={On the $L_p$ dual Minkowski problem},
	journal={Adv. Math.},
	volume={332},
	date={2018},
	pages={57--84},
	issn={0001-8708},
	review={\MR{3810248}},
	doi={10.1016/j.aim.2018.05.002},
}




\bib{IM}{article}{
   author={Ivaki M.},
   title={Deforming a hypersurface by principal radii of curvature and support function},
   journal={Calc. Var. PDEs},
   volume={58(1)}
   date={2019},
}


\bib{JLL}{article}{
	author={Ju H.},
	author={Li B.},
	author={Liu Y.},
	title={Deforming a convex hypersurface by anisotropic curvature flows},
	journal={Adv. Nonlinear Stud.},
	volume={21},
	date={2021},
	number={1},
	pages={155--166},
	issn={1536-1365},
	review={\MR{4234083}},
	doi={10.1515/ans-2020-2108},
}



\bib{KNV}{book}{
  author={Krylov N. V.},
     title= {Nonlinear elliptic and parabolic quations of the second order},
 publisher={D. Reidel Publishing Co., Dordrecht},
     date={1987. xiv+462pp},

}



\bib{LN}{book}{
  author={Nirenberg L.},
     title= {On a generalization of quasi-conformal mappings and its application to elliptic partial differential equations},
 publisher={Contributions to the theory of partial differential equations, Annals of Mathematics Studies},
     date={ Princeton University Press, Princeton, N. J.,1954, pp. 95C100.}
  }



\bib{LSW}{article}{
   author={Li Q.},
   author={Sheng W.},
   author={Wang X-J},
   title={Flow by Gauss curvature to the Aleksandrov and dual Minkowski problems},
   journal={Journal of the European Mathematical Society},
   volume={22}
   date={2019},
   pages={893-923},
}

\bib{LL}{article}{
	author={Liu Y.},
	author={Lu J.},
	title={A flow method for the dual Orlicz-Minkowski problem},
	journal={Trans. Amer. Math. Soc.},
	volume={373},
	date={2020},
	number={8},
	pages={5833--5853},
	issn={0002-9947},
	review={\MR{4127893}},
	doi={10.1090/tran/8130},
}

\bib{LE}{article}{
	author={Lutwak E.},
	title={The Brunn-Minkowski-Firey theory. I. Mixed volumes and the
		Minkowski problem},
	journal={J. Differential Geom.},
	volume={38},
	date={1993},
	number={1},
	pages={131--150},
	issn={0022-040X},
	review={\MR{1231704}},
}



\bib{LO}{article}{
	author={Lutwak E.},
	author={Oliker V.},
	title={On the regularity of solutions to a generalization of the
		Minkowski problem},
	journal={J. Differential Geom.},
	volume={41},
	date={1995},
	number={1},
	pages={227--246},
	issn={0022-040X},
	review={\MR{1316557}},
}


\bib{LYZ2}{article}{
	author={Lutwak E.},
	author={Yang D.},
	author={Zhang G.},
	title={$L_p$ dual curvature measures},
	journal={Adv. Math.},
	volume={329},
	date={2018},
	pages={85--132},
	issn={0001-8708},
	review={\MR{3783409}},
	doi={10.1016/j.aim.2018.02.011},
}

\bib{SR}{book}{
	author={Schneider R.},
	title={Convex bodies: the Brunn-Minkowski theory},
	series={Encyclopedia of Mathematics and its Applications},
	volume={151},
	edition={Second expanded edition},
	publisher={Cambridge University Press, Cambridge},
	date={2014},
	pages={xxii+736},
	isbn={978-1-107-60101-7},
	review={\MR{3155183}},
}







\bib{UJ}{article}{
   author={Urbas J.},
   title={An expansion of convex hypersurfaces},
   journal={J. Diff. Geom.},
   volume={33(1)}
   date={1991},
   pages={91-125},
}




\bib{Z}{article}{
	author={Zhao Y.},
	title={The dual Minkowski problem for negative indices},
	journal={Calc. Var. Partial Differential Equations},
	volume={56},
	date={2017},
	number={2},
	pages={Paper No. 18, 16},
	issn={0944-2669},
	review={\MR{3605843}},
	doi={10.1007/s00526-017-1124-x},
}





\end{biblist}
\end{document}